\documentclass{amsart}

\usepackage{amssymb,amsmath,amscd}
\usepackage{color}

\theoremstyle{definition}
\newtheorem{definition}{Definition}[section]
\newtheorem{remark}[definition]{Remark}

\theoremstyle{plain}

\newtheorem{Proposition}[definition]{Proposition}
\newtheorem{Theorem}[definition]{Theorem}

\newtheorem{lemma}[definition]{Lemma}
\newtheorem{proposition}[definition]{Proposition}

\newtheorem{corollary}[definition]{Corollary}

\usepackage[a4paper,hscale=0.8,vscale=0.8,hmarginratio=1:1,includehead,headheight=16.5pt,margin=2.7cm]{geometry}
\usepackage{fancyhdr}
\usepackage{amssymb,amsmath,amscd,dsfont}
\usepackage{slashed}
\usepackage{mathtools}
\newenvironment{tabsection}{}{}

\newcommand{\loc}{\ensuremath{ \op{loc}}}

\newcommand{\Lf}{\ensuremath{\mathbf{L}}}
\newcommand{\argument}{\ensuremath{\mathinner{\;\cdot\;}}}
\newcommand{\op}[1]{\ensuremath{\operatorname{#1}}}
\newcommand{\wt}[1]{\ensuremath{\widetilde{#1}}}
\newcommand{\wh}[1]{\ensuremath{\widehat{#1}}}

\newcommand{\ul}[1]{\ensuremath{\underline{#1}}}

\newcommand{\cG}{\ensuremath{\mathcal{G}}}

\newcommand{\cV}{\ensuremath{\mathcal{V}}}

\newcommand{\fz}{\ensuremath{\mathfrak{z}}}

\newcommand{\fk}{\ensuremath{\mathfrak{k}}}

\newcommand{\fg}{\ensuremath{\mathfrak{g}}}

\newcommand{\fh}{\ensuremath{\mathfrak{h}}}

\newcommand{\mf}[1]{\ensuremath{\mathfrak{#1}}}

 \newcommand{\bC}{\ensuremath{\mathbb{C}}}
 
 \newcommand{\R}{\ensuremath{\mathbb{R}}}

 \newcommand{\Z}{\ensuremath{\mathbb{Z}}}
 
 \newcommand{\K}{\ensuremath{\mathbb{K}}}

\newcommand{\bS}{\ensuremath{\mathbb{S}}}

\newcommand{\id}{\ensuremath{\operatorname{id}}}
\newcommand{\pr}{\ensuremath{\operatorname{pr}}}
\newcommand{\ev}{\ensuremath{\operatorname{ev}}}

\newcommand{\ad}{\ensuremath{\operatorname{ad}}}

\newcommand{\Obj}{\ensuremath{\operatorname{Obj}}}

\newcommand{\out}{\ensuremath{\operatorname{out}}}

\DeclareMathOperator{\Lift}{Lift}

\newcommand{\cat}[1]{\ensuremath{\mathsf{\mathop{#1}}}}

\newcommand{\der}{\ensuremath{\operatorname{der}}}

\newcommand{\Diff}{\ensuremath{\operatorname{Diff}}}

\newcommand{\Hilb}{\ensuremath{\operatorname{Hilb}}}

\newcommand{\im}{\ensuremath{\operatorname{im}}}
\newcommand{\coker}{\ensuremath{\operatorname{coker}}}

\newcommand{\Spin}{\ensuremath{\operatorname{Spin}}}

\newcommand{\End}{\ensuremath{\operatorname{End}}}

\newcommand{\se}{\ensuremath{\nobreak\subseteq\nobreak}}
\newcommand{\from}{\ensuremath{\nobreak\colon\nobreak}}
\renewcommand{\to}{\ensuremath{\nobreak\rightarrow\nobreak}}

\newcommand{\nbinto}{\ensuremath{\nobreak\hookrightarrow\nobreak}}

\newcommand{\sprod}{\ensuremath{\mathopen{\langle}\mathinner{\cdot},\mathinner{\cdot}\mathclose{\rangle}}}

\newcommand{\toto}{\ensuremath{\nobreak\rightrightarrows\nobreak}}

\def\raisefix#1{%
  \ifx#1\displaystyle
    \raise.39ex
  \else
    \ifx#1\textstyle
      \raise.39ex
    \else
      \ifx#1\scriptstyle
        \raise.275ex
      \else
        \raise.150ex
      \fi
    \fi
  \fi
}
\def\stylefix#1{%
  \ifx#1\displaystyle
    \scriptstyle
  \else
    \ifx#1\textstyle
      \scriptstyle
    \else
      \ifx#1\scriptstyle
        \scriptscriptstyle
      \else
        \scriptscriptstyle
      \fi
    \fi
  \fi
}

\DeclareFontFamily{U}{mathx}{\hyphenchar\font45}
\DeclareFontShape{U}{mathx}{m}{n}{
      <5> <6> <7> <8> <9> <10>
      <10.95> <12> <14.4> <17.28> <20.74> <24.88>
      mathx10
      }{}
\DeclareSymbolFont{mathx}{U}{mathx}{m}{n}
\DeclareFontSubstitution{U}{mathx}{m}{n}
\DeclareMathAccent{\widecheck}{0}{mathx}{"71}
\DeclareMathAccent{\wideparen}{0}{mathx}{"75}

\usepackage{hyperref}
\hypersetup{
 urlcolor = red,
 colorlinks = true,
 linkcolor = blue,
 citecolor = blue,
 linktocpage = true,
 pdftitle = {Quasi-periodic paths and a string 2-group model from the free loop group},
 pdfauthor = {Michael Murray, David Roberts, Christoph Wockel},
 bookmarksopen = true,
 bookmarksopenlevel = 1,
 unicode = true,
 hypertexnames =false 
}

\usepackage[notcite,notref,final]{showkeys}

\usepackage{enumitem}
\setlist[enumerate]{label={\alph*})}

\usepackage{xspace}

\usepackage{ifpdf}
\ifpdf
\usepackage[all,pdf]{xy} %% NB this MUST be loaded as the last package.
\else
\input xy
\xyoption{all}
\xyoption{2cell}
\xyoption{v2}
\fi
\UseComputerModernTips
\newcommand{\xyhookrightarrow}[1]{\ar@{}[r]|-*[@]{\xhookrightarrow{\hphantom{\hspace{#1}}}}}

\providecommand{\cat}[1]{\mathsf{\mathop{#1}}}
\newcommand{\cohLie}{\cat{cLie}$-$\cat{2Grp}}
\newcommand{\LieGrpd}{\cat{LieGrpd}}
\newcommand{\Vect}{\cat{Vect}}
\newcommand{\Rep}{\cat{Rep}}
\newcommand{\Cat}{\cat{Cat}}
\newcommand{\canCocycle}{\ensuremath{\mathopen{\langle}\mathinner{\cdot},\mathopen{[}\mathinner{\cdot},\mathinner{\cdot}\mathclose{]}\mathclose{\rangle}}}
\newcommand{\cent}{\ensuremath{\op{cent}}}

\newcommand{\etalchar}[1]{$^{#1}$}

\begin{document}

\title[Quasi-periodic paths and a string 2-group model]{Quasi-periodic paths and a string 2-group model from the free loop group}
\thanks{The authors acknowledge support under the  Australian Research Council's \emph{Discovery Projects} funding scheme (project numbers DP120100106 and DP130102578)}

\author{Michael Murray}
\address[M.K.\ Murray]{School of Mathematical Sciences, 
University of Adelaide,
Adelaide, SA 5005, 
Australia}
\email{michael.murray@adelaide.edu.au}

\author[D.M.\ Roberts]{David Michael Roberts}
\address[D.M.\ Roberts]{School of Mathematical Sciences, 
University of Adelaide,
Adelaide, SA 5005, 
Australia}
\email{david.roberts@adelaide.edu.au}

\author[C.\ Wockel]{Christoph Wockel}
\address[C.\ Wockel]
{Department of Mathematics, University of Hamburg (at the time of writing)}
\email{christoph@wockel.eu}



\begin{abstract}
In this paper we address the question of the existence of a model for the
 string 2-group as a strict Lie-2-group using the free loop group $L\Spin$
  (or more generally $LG$ for compact simple simply-connected Lie groups $G$).
   Baez--Crans--Stevenson--Schreiber constructed a model for
 the string 2-group using a based loop group.
 This has the deficiency that it does not admit an action of the circle group
 $S^{1}$, which is of crucial importance, for instance in the construction of a
 (hypothetical) $S^{1}$-equivariant index of (higher) differential operators.
 The present paper shows that there are in fact obstructions for constructing a
 \emph{strict} model for the string 2-group using $LG$. We show that a  certain infinite-dimensional 
 manifold of smooth paths admits no Lie group structure, and that there are no nontrivial Lie crossed 
 modules analogous to the BCSS model using the universal central extension of the \emph{free} loop group. 
 Afterwards, we construct the next best thing, namely a  \emph{coherent} model for the string 2-group using the free loop group, with explicit formulas for all structure.
 This is in particular important for the expected representation theory of the string group that we discuss briefly in the end.
\end{abstract}

\maketitle

\section*{Introduction} 

The string group is one of the key examples, and motivation from applications, of higher differential geometry. 
It was originally conceived as a homotopical object: the 3-connected cover of a given compact simple simply-connected Lie group $G$ (and named by Haynes Miller). There is no finite-dimensional Lie group that can be this cover, since such a group necessarily has, as kernel of the map to $G$, a closed subgroup that has non-trivial second homotopy group. Since closed subgroups of finite-dimensional Lie groups are automatically Lie subgroups (and clearly finite-dimensional!), this contradicts Cartan's famous result that each finite-dimensional Lie group has vanishing second homotopy group.
One is thus forced to seek more exotic smooth objects to play the r\^ole of the string group in geometry, and there have been several different constructions of \emph{models} of the string group, namely by smooth categorical groups, or \emph{Lie 2-groups}. 

There have been models by a strict Lie 2-group (or equivalently a Lie crossed module) \cite{BaezCransStevensonSchreiber07From-loop-groups-to-2-groups}; 
by a simplicial Banach manifold after higher Lie integration of an $L_\infty$-algebra \cite{Henriques08Integrating-Loo-algebras}; by a finite-dimensional `stacky' Lie group \cite{Schommer-Pries10Central-Extensions-of-Smooth-2-Groups-and-a-Finite-Dimensional-String-2-Group}; 
by a differentiable group stack in cohesive $\infty$-topos theory (eg \cite{FiorenzaSchreiberStasheff12Cech-cocycles-for-differential-characteristic-classes}); 
by a transgression/regression of the Gawedzki--Reis/Meinrenken basic gerbe \cite{Waldorf12A-construction-of-string-2-group-models-using-a-transgression-regression-technique}; 
and by an infinite-dimensional plain Lie group, as well as a Lie 2-group involving this group \cite{NikolausSachseWockel13A-smooth-model-for-the-string-group}.
One should think of all of these models as analogous to different choices of coordinates on a higher-geometric object. 

What all these models have in common is that they are tuned to a particular application, or are displaying a particular desirable feature at the cost of losing out in other ways.
For instance, most of the above models are infinite-dimensional, and are built from Fr\'echet manifolds. Henriques' model uses Banach manifolds, which are better behaved, but is only a reduced simplicial manifold, rather than an actual groupoid with multiplication functor. 
The finite-dimensional model of Schommer-Pries has only a multiplication anafunctor/bibundle, rather than an honest functor of Lie groupoids.
The group stack of Schreiber et al.\ is defined entirely canonically (in fact by a universal property), but is an object of a higher topos and so needs a choice of presentation in order to do hard geometry in local coordinates.

It is to this list of models we wish to offer one more, built using the free smooth loop group. 
Recall that \cite{BaezCransStevensonSchreiber07From-loop-groups-to-2-groups} consider the basic gerbe on a compact simple simply-connected Lie group $G$, 
using the path fibration and the group of  smooth maps $[0,1] \to G$ with both endpoints being mapped to the identity.  
One desideratum for using the free loop group is that it has a smooth action of the circle group by rotation.
Indeed, the origin of \emph{string structures} in the 1980s (since reinterpreted using the string group) in work of Killingback used loop spaces and bundles for loop groups, and Witten's subsequent work \cite{Witten87Elliptic-genera-and-quantum-field-theory} on the topic considered a hypothetical circle-equivariant Dirac operator on a loop space $LM$ of a manifold $M$ and its index.
The definition of such an operator is extremely difficult, and Taubes' rigorous construction \cite{Taubes89S1-actions-and-elliptic-genera} only worked over a formal neighbourhood of the constant loops inside the full loop space (in fact, on the normal bundle to $M\subset LM$).
Having a string group-related construction that linked in any way to Witten's physico-mathematical arguments is thus very desirable.

The necessity of a meaningful $S^1$-action on (one of) the `ingredients' of a string group model in question naturally leads to the question whether there exists a strict Lie 2-group model that has the full loop group $LG$ built in.
If one considers a model as being like a choice of coordinates, the existence of a circle action is like choosing a particularly symmetric system of charts.
Originally, the authors tried to build this model but did not succeed. 
After a while it became clear that there is some obstruction for this, which is manifested in the relatively small outer automorphism group of the full loop group. 
The latter is just the diffeomorphism group of $S^1$ (times a discrete group), which is homotopy equivalent to the subgroup of rigid rotations and hence to $S^1$ itself. 
Since crossed modules are basically pull-backs of homomorphisms into the outer automorphism group, this shows (at least morally) that such a homomorphism cannot change the homotopy type in such a way that the third homotopy group gets killed.
This argument is presented in detail below and shows that there \emph{cannot} exist a model for the string group as a Lie 2-group that is built from the full loop group.
More precisely, we prove that there is no nontrivial crossed module $\tau\from \wh{LG}\to H$ with kernel $U(1)$ and cokernel $G$ (for any Lie group $H$). 
This is in contrast to the BCSS crossed module $\wh{\Omega G} \to PG$, which is nontrivial and with the given kernel/cokernel.

Instead of insisting on a strict Lie 2-group built using $LG$, we instead found a Lie groupoid, here denoted $S(QG)$ (see Definition~\ref{def:SQG}), that is equivalent to the underlying Lie groupoid of a different strict Lie 2-group (see Definition~\ref{def:QstarflatQG}). 
This strict 2-group is morally similar to the BCSS strict 2-group model \cite{BaezCransStevensonSchreiber07From-loop-groups-to-2-groups}, but instead of using the full based path space $PG$, we use a Fr\'echet manifold diffeomorphic to the submanifold of $PG$ consisting of based paths that are \emph{flat} at both endpoints.
Much as a diffeomorphism of a manifold $M$ with (the manifold underlying) a Lie group induces a unique Lie group structure of $M$, we use the equivalence mentioned above to induce an essentially unique 2-group structure on $S(QG)$.
Since we are transferring structure across an \emph{equivalence} and not an isomorphism, the transferred 2-group structure is not strict, but merely \emph{coherent}; the usual group axioms only hold up to specified isomorphism, and these isomorphisms satisfy coherence axioms.
We give explicit formulas for the structure maps of the 2-group, given a choice of a smoothed step function (which can itself be written down explicitly, if so chosen).

In summary we proceed as follows.  Section 1 explains the differential geometry of the space of quasi-periodic paths. In Section 2 we prove
the non-existence of \emph{strict} Lie $2$-group models for the string group using the central extension of the free loop group. The final Section 3 explains how to contruct an explicit coherent Lie $2$-group model, and show that our method gives an essentially unique 2-group structure on the Lie groupoid $S(QG)$.  We conclude with some comments on 2-representations of the coherent $2$-group that we have constructed, including some conjectural remarks on the relation to the Freed--Hopkins--Teleman theorem.

 The authors are grateful to Karl-Hermann Neeb for pointing out the reference
 \cite{Hofmann65Lie-algebras-with-subalgebras-of-co-dimension-one} in the
 classification of infinitesimal actions of semi-simple Lie algebras on the
 circle. DMR thanks Chris Schommer-Pries for providing the helpful answer \cite{SchommerPriesMO} to his MathOverflow question, and those who responded on the \texttt{categories} mailing list to correct a misunderstanding about lifting adjunctions. The authors also thank the referee for useful comments and corrections.

\section*{Notation and Conventions} %

\begin{tabsection}
 A manifold (respectively Lie group) will throughout be a topological space
 modelled on open subsets of a locally convex vector space with smooth chart
 changes in the sense of
 \cite{Neeb06Towards-a-Lie-theory-of-locally-convex-groups}. A Lie groupoid
 will be a groupoid with surjective submersions\footnote{See \cite{Glockner15Fundamentals-of-submersions-and-immersions-between-infinite-dimensional-manifolds} for an overview of submersion properties for locally convex manifolds. Most importantly, this means that submersions are locally split on charts, not merely on tangent spaces.} as source and target map and
 smooth structure maps otherwise. These are the objects of the strict
 2-category $\LieGrpd$ with morphisms the smooth functors and 2-morphisms the
 smooth natural transformations
 \cite{Wockel13Infinite-dimensional-and-higher-structures-in-differential-geometry}.
 A coherent Lie 2-group is a coherent 2-group in $\LieGrpd$ in the sense of
 \cite[Definition 19]{BaezLauda04Higher-dimensional-algebra.-V.-2-groups}.
 These are the objects of the strict 2-category $\cohLie$ of coherent 2-groups
 in $\LieGrpd$ (cf.\ \cite[Section
 7]{BaezLauda04Higher-dimensional-algebra.-V.-2-groups}). If $G$ is a simple
 compact and 1-connected Lie group with Lie algebra $\fg$, then we denote by
 $\sprod\from S^{2}\fg\to \R$ a multiple of the Killing form on $\fg$ which makes the left invariant form 
 $\omega\in \Omega^{3}(G)$  with value
 $\canCocycle$ on $\Lambda^{3}\fg$ a generator of $H^3(G, \Z)$.    We will always identify $S^{1}$ with
 $\R/\Z$ and $U(1)$ with $\{z\in\bC\mid z\bar{z}=1\}$. The base-point of $G$ is $e$ or $1$,
 of $U(1)$ is $1$ and of $S^{1}$ is $0+\Z$.
\end{tabsection}

\section{The geometry of quasi-periodic paths} %
\label{sec:path_groups_loop_groups_and_central_extensions}

\begin{tabsection}
 In this section we collect some background material that is necessary for the
 constructions later on. In particular, we discuss the path-loop fibration of
 $G$ consisting of smooth paths and loops (whereas the loops in the model from
 \cite{BaezCransStevensonSchreiber07From-loop-groups-to-2-groups} are everywhere continuous, but only smooth away from the base-point).  What we gain from this in the end will be a circle action
 on the free loop group $LG$. This will be important when talking about
 representations later on.
 
 Unless mentioned otherwise, $G$ may be an arbitrary (possibly  
 infinite-dimensional) Lie group throughout this section.
\end{tabsection}

\begin{definition}
 The free loop group $LG$ of $G$ is the space of periodic paths
 \begin{equation*}
  LG:=\{\eta\in C^{\infty}(\R,G)\mid \gamma(t+1)=\gamma(t)\text{ for all }t\in \R\}\cong C^{\infty}(S^{1},G).
 \end{equation*}
 We endow this with the topology of uniform convergence of all derivatives on
 compact subsets, which turns it into a Fr\'echet--Lie group
 \cite{Glockner02Lie-group-structures-on-quotient-groups-and-universal-complexifications-for-infinite-dimensional-Lie-groups,Wockel13Infinite-dimensional-and-higher-structures-in-differential-geometry}.
 The quasi-periodic paths are given by
 \begin{equation*}
  QG:=\{\gamma\in C^{\infty}(\R,G)\mid \gamma(t+1)\cdot \gamma(t)^{-1}\text{ is constant}\},
 \end{equation*}
 which is for now only a set (we will endow it with a smooth structure below).
 This is the natural space of ``smooth paths'' in $G$ that has $LG$ as the
 subspace of ``periodic paths''. 
\end{definition}

This definition has appeared in \cite[Proof of Proposition~5.1]{AtiyahSegal04Twisted-K-theory} for the special case of $G$ being compact and in \cite{AlekseevMeinrenken08The-Atiyah-algebroid-of-the-path-fibration-over-a-Lie-group} for $G$ being finite-dimensional.
Notice also that if one defines the left logarithmic derivative by 
\begin{equation} \label{eq:leftlogder}
\delta^{l}\from QG\to \Omega^{1}(\R,\fg),\quad \gamma\mapsto \delta^{l}(\gamma)=\gamma^{-1}\cdot d\gamma
\end{equation}
then $\gamma \in QG$ precisely when $\delta^l(\gamma)$ is periodic, i.e., $\delta^l(\gamma)\in\Omega^{1}(\R,\fg)^{\Z}\cong \Omega^{1} (S^{1},\fg)$. 
If $G$ is Milnor regular the map $QG  \to \Omega^1(S^1,\fg) \times G$ obtained from the left logarithmic derivative and the initial value is a diffeomorphism.  
Our results however hold in more generality unless stated, so we do not use this fact.

\begin{remark}\label{rem:QG}
 For $\gamma\in QG$ the constant $\gamma(t+1)\cdot \gamma(t)^{-1}$ is given by
 $\gamma(1)\cdot \gamma(0)^{-1}$ and this  yields a map
 \begin{equation*}
  \Theta\from QG\to G,\quad \gamma\mapsto \gamma(1)\cdot\gamma(0)^{-1}.
 \end{equation*}
 Thus we have $\gamma(t+1)=\Theta(\gamma)\cdot \gamma(t)$ for $\gamma\in Q$.
 However, note that $QG$ does not carry a natural multiplication. If we
 multiply paths point-wise, then the result might not be quasi-periodic.
 Indeed, a necessary condition for $\gamma\in C^{\infty}(\R,G)$ to be
 quasi-periodic is that
 \begin{equation*}
  \dot \gamma (1)= \gamma(1)\cdot \gamma(0)^{-1}\cdot \dot \gamma(0).
 \end{equation*}
 This is not satisfied for the product of two arbitrary quasi-periodic paths.
 \end{remark}

 We shall now discuss the smooth structure on $QG$. The topology on $QG$ is the
 topology of uniform convergence of all derivatives on compact subsets, which
 renders $\Theta$ continuous. Clearly, the fibre of $\Theta$ over $1$ is $LG$.

\begin{Proposition}\label{prop:QG_is_principal_LG_bundle}
 The map $\Theta\from QG\to G$ makes $QG$ a continuous principal $LG$-bundle
 over $G$, where $LG$ acts by right multiplication  on $QG$. Moreover, the trivialisation changes of this bundle are smooth, giving rise to a smooth structure on $QG$ turning $\Theta\from QG\to G$ into a smooth principal $LG$-bundle.
\end{Proposition}

We then can give $QG$ the structure of a smooth manifold using the local triviality of this bundle.
\begin{proof}
 To this end, we first note that $LG$ acts continuously on $QG$ from the right
 via
 \begin{equation*}
  QG\times LG\to QG,\quad (\gamma,\eta)\mapsto \gamma\cdot \eta.
 \end{equation*}
 Now let $U\se G$ be an open identity neighbourhood and  let
 $\varphi\from U\to \varphi(U)\se \fg$ be a chart with $\varphi(1)=0$ and
 $\varphi(U)$ convex. Choose a smooth $\rho \colon [0, 1] \to [0, 1]$ with $\rho(0) = 0$, $\rho(1) = 1$ and  all derivatives of $\rho$ vanishing at $0$ and $1$.  Then for each $g\in G$ define a  smooth map
 \begin{equation*}
  F_{g}\from [0,1]\times gU\to G\quad \text{ by  }\quad F_g(t, x) = g \varphi^{-1}(\rho(t) \varphi(g^{-1} x)) .
 \end{equation*}
All derivatives of the curves $\gamma_{g,x}\from [0,1]\to G$,
 $t\mapsto F_{g}(t,x)$ vanish for $t=0$ and $t=1$ and all $x\in gU$ (we will
 usually suppress the subscript $g$ on $\gamma_{g,x}$ in the notation). Thus
 $\gamma_{x}$ may be extended to a quasi-periodic path $\R \to G$ satisfying
 $\gamma_{x}(t+1)=x\cdot \gamma_{x}(t)$. Moreover, $gU\to QG$,
 $x\mapsto \gamma_{x}$ is a continuous section of $\Theta$. This gives rise to
 a trivialisation
 \begin{equation*}
  \Phi_{F_{g}}\from  \Theta^{-1}(gU)\to gU\times LG,\quad \gamma\mapsto (\Theta(\gamma),\gamma_{\Theta(\gamma)}^{-1}\cdot \gamma).
 \end{equation*}
 Note that $\gamma_{\Theta(\gamma)}^{-1}\cdot \gamma$ is smooth since all
 derivatives of ${\gamma_{\Theta(\gamma)}}$ vanish for all $t\in\Z$. Moreover,
 \begin{align*}
  (\gamma_{\Theta(\gamma)}^{-1}\cdot \gamma)(t)= & \ \gamma_{\Theta(\gamma)}^{-1}(t)\cdot \gamma(t)=\gamma_{\Theta(\gamma)}^{-1}(t)\cdot \Theta(\gamma)^{-1}\cdot\gamma(t+1)\\
 = & \ \gamma_{\Theta(\gamma)}^{-1}(t+1)\cdot\gamma(t+1)=
  (\gamma_{\Theta(g)}^{-1}\cdot \gamma)(t+1)
 \end{align*}
 implies that $\gamma_{\Theta(\gamma)}^{-1}\cdot \gamma$ actually is an element
 of $LG$. An inverse of $\Phi_{F_{g}}$ is given by
 $(x,\eta)\mapsto \gamma_{x}\cdot \eta$. The trivialisation changes are then
 given by
 \begin{equation*}
  \Phi_{F_{h}}\circ  \Phi_{F_{g}}^{-1}\from  (gU\cap hU)\times LG,\quad (x,\eta)\mapsto \gamma_{h,x}\cdot \gamma_{g,x}^{-1}\cdot \eta,
 \end{equation*}
 which are clearly smooth and commute with the right multiplication action of
 $LG$ on itself. Thus there exists a unique smooth structure on $QG$ turning
 each $\Phi_{F_{g}}$ into a diffeomorphism. This smooth structure then turns
 $\Theta\from QG\to G$ into a smooth principal $LG$-bundle.
\end{proof}
 
 Note the smooth structure on $QG$ is uniquely determined by the underlying topological
 principal bundle
 \cite[CorollaryII.13]{MullerWockel07Equivalences-of-Smooth-and-Continuous-Principal-Bundles-with-Infinite-Dimensional-Structure-Group}. It follows that the charts of $QG$ are modelled on the vector space underlying the Lie algebra $L\fg\times\fg$.

 From the fact that $\Theta\from QG\to G$ is a principal bundle it follows that
 the action map $QG\times LG\to QG\times_{G}QG$ a diffeomorphism. Moreover, in
 this particular situation the smooth inverse of the action map is given by
 \begin{equation*}
  QG\times _{G} QG\to QG\times LG, \quad (\gamma,\eta)\mapsto (\gamma,\gamma^{-1}\cdot\eta),  
 \end{equation*}
 where $\gamma^{-1}\cdot\eta$ denotes the point-wise product of smooth paths
 in $G$ (note that $\gamma^{-1}\cdot\eta$ is indeed periodic since
 $\Theta(\gamma)=\Theta(\eta)$). Since $G$ has a natural base-point,
 $\Theta$ has a natural fibre which is, of course, the space of loops $LG$ in
 $QG$. From this we obtain a natural principal $LG$-bundle
 \begin{equation}\label{eqn:quasi-periodic-path-loop-fibration}
  LG\hookrightarrow QG \xrightarrow{\Theta}G
 \end{equation}
 over $G$. We can view $QG$ also as a bundle for the smooth based loop group
 \begin{equation*}
  \Omega G:=\{\gamma\in LG\mid \gamma(0)=e \}=\{\gamma \in QG\mid \gamma(0)=\gamma(1)=e\}.
 \end{equation*} 
 Indeed, the evaluation map $\ev_{0,1}\from QG\to G\times G$,
 $\gamma\mapsto (\gamma(0),\gamma(1))$ has $\Omega G$ as canonical fibre over $(e,e)\in G\times G$.
 
 Let us consider the quotient bundle
 \begin{equation*}
  QG/G=QG\times_{LG}LG/G\to G,
 \end{equation*}
 which is obtained by taking the bundle associated to the smooth action by left
 translations on the homogeneous space $LG/G$. Note that
 \begin{equation*}
   LG/G\to \Omega G,\quad [\gamma]\mapsto \gamma\cdot \gamma(0)^{-1}
 \end{equation*}
 is an equivariant diffeomorphism if we endow $\Omega G$ with the $LG$ action
 \begin{equation*}
  LG\times \Omega G\to \Omega G,\quad \gamma.\eta:=\gamma\cdot \eta\cdot \gamma(0)^{-1}
 \end{equation*}
 from the left.
 
 This way we obtain a locally trivial $LG/G$ fibre bundle over $G$. 
If
 $P_{c}G:=\{\gamma\in C([0,1],G)\mid \gamma(0)=e\}$ denotes the continuous
 pointed paths in $G$, then there is a continuous map
 \begin{equation*}
  \Xi\from  QG/G\to P_{c}G,\quad [\gamma]\mapsto \left.(\gamma\cdot \gamma(0)^{-1})\right|_{[0,1]}.
 \end{equation*}
 If $\Omega_{c}G:=\{\gamma\in C([0,1],G)\mid \gamma(0)=\gamma(1)=e\}$ denotes
 the continuous pointed loop group, and we identify $\Omega G$ with $LG/G$
 via the $LG$-equivariant diffeomorphism $\Omega G\to LG/G$,
 $\gamma\mapsto [\gamma]$, then $\Xi$ restricts on $\Omega G$ to the obvious
 map $\Omega G\to \Omega_{c}G$, $\gamma\mapsto \left.\gamma\right|_{[0,1]}$.
 Thus we obtain a commuting diagram
 \begin{equation}\label{eqn4}
  \vcenter{  \xymatrix{
  \Omega G\cong LG/G	 \xyhookrightarrow{1.8em}  \ar[d]^{\left.{\Xi}\right|_{\Omega G}} & QG/G\ar[d]^{\Xi} \ar[r]^{\Theta} & G\ar@{=}[d]& \\
  \Omega_{c}G \xyhookrightarrow{4em} & P_{c}G\ar[r]^{\ev_{1}} & G.
  }}
 \end{equation}
 Since $\left.\Xi\right|_{\Omega G}$ is a weak homotopy equivalence, so is
 $\Xi$ and since $P_{c}G$ is contractible it follows that $QG/G$ is also
 contractible \cite{Palais66Homotopy-theory-of-infinite-dimensional-manifolds}
 (that $QG/G$ is metrisable follows from the metrisability of $\Omega G$ and
 $G$ \cite{EtterGriffin54On-the-metrizability-of-the-bundle-space,Wockel13Infinite-dimensional-and-higher-structures-in-differential-geometry}).

Another way to see that $QG/G$ is contractible is to observe that the 
left logarithmic derivative \eqref{eq:leftlogder}  
 factors through a diffeomorphism $QG/G\xrightarrow{\simeq}
\Omega^{1}(S^{1},\fg)$ \cite[Corollary
2.3]{NeebWagemann08Lie-group-structures-on-groups-of-smooth-and-holomorphic-maps-on-non-compact-manifolds}. The space $\Omega^{1}(S^{1},\fg)$ is the space of connections on the trivial $G$-bundle over $\bS^{1}$, which has also been used as a model for the path fibration, for instance in \cite{CarMic}.
 
 From the
 long exact homotopy sequences induced by \eqref{eqn4} and from
 $\pi_{n}(LG)\cong \pi_{n+1}(G)\oplus \pi_{n}(G)$ one can also read off the
 short exact sequences
 \begin{equation*}
  0 \to \pi_{n+1}(G)\to \pi_{n+1}(G)\oplus \pi_{n}(G) \to \pi_{n}(QG)\to 0
 \end{equation*}
 and thus $\pi_{n}(QG)\cong \pi_{n}(G)$.

\begin{tabsection}
 Although the bundle $\Theta\from QG\to G$ arises quite naturally, it is
 unfortunately not very useful for Lie theoretic arguments since $QG$ does not
 admit a Lie group structure that turns $QG$ into an extension of $G$ by $LG$
 (whereas $QG$ \emph{is} a principal $LG$-bundle over $G$). This is what we will show
 now.
\end{tabsection}

\begin{proposition}\label{prop:extensions_of_current_algebras}
 Suppose $M$ is a compact closed manifold and $\fk,\fg$ are finite-dimensional Lie algebras. 
 If $\fg$ is a simple real Lie algebra such that $\fg_{\bC}$ is a simple complex Lie algebra, then each Lie algebra  extension $C^{\infty}(M,\fg)\to \wh{\fk}\to\fk$ is isomorphic to the semi-direct product
 \begin{equation*}
  C^{\infty}(M,\fg)\to C^{\infty}(M,\fg)\rtimes \fk\to\fk,
 \end{equation*}
 where $\fk$ acts on $C^{\infty}(M,\fg)$ via an action of $\fk$ on $M$; that is, a Lie algebra homomorphism $\fk\to\cV(M)$ composed with the natural action $\cV(M)\to\der (C^{\infty}(M,\fg))$).
\end{proposition}

\begin{proof}
 We will use the extension theory of topological Lie algebras
 \cite{Neeb06Non-abelian-extensions-of-topological-Lie-algebras} to show that
 each extension gives rise to an action $\fk\to\cV(M)$ and that this extension
 then is isomorphic to $C^{\infty}(M,\fg)\rtimes\fk$.
 
 If we fix a linear section $s\from \fk\to \wh{\fk}$, then
 $x\mapsto [s(x),\argument]$ gives rise to a linear map
 $\fk\to \der(C^{\infty}(M,\fg))$. Since
 $$\fz(C^{\infty}(M,\fg))=C^{\infty}(M,\fz(\fg))=\{0\}$$
 we can identify $C^{\infty}(M,\fg)$ with a subalgebra of
 $\der(C^{\infty}(M,\fg))$ and the composition of
 $\fk\to\der{C^{\infty}(M,\fg)}$ with the quotient map
 $$
 \der(C^{\infty}(M,\fg))\to \out(C^{\infty}(M,\fg)):=\der(C^{\infty}(M,\fg))/C^{\infty}(M,\fg)
 $$
 gives rise to a Lie algebra homomorphism
 $r\from \fk\to\out(C^{\infty}(M,\fg))$.
 
 Recall that the centroid of a Lie algebra $\fh$ is the commutant of the
 adjoint representation
 \begin{equation*}
  \cent(\fh):=\{f\in \End(\fh)\mid \ad(x)\circ f=f \circ\ad(x)\text{ for all }x\in\fh \}.
 \end{equation*}
 Clearly, $\cent(\fh)$ is an associative $\K$-subalgebra of $\End(\fh)$ if
 $\fh$ is a Lie algebra over $\K$. Moreover, each derivation
 $D\in\der(\fh)$ induces a derivation 
 $$
 \cent(D)\from \cent(\fh)\to \cent(\fh)
 $$
 of $\K$-algebras via the assignment $f\mapsto f \circ D-D \circ f$. A
 straightforward calculation shows that
 $\cent\from \der(\fh)\to \der(\cent(\fh))$ is a Lie algebra homomorphism that
 has the inner derivations in the kernel.
 
 In the case $\fh=C^{\infty}(M,\fg)$ this induces a morphism of Lie algebras
 $$
 \out(C^{\infty}(M,\fg))\to \der(\cent(C^{\infty}(M,\fg))).
 $$ By \cite[Theorem
 2.76]{Gundogan09Lie-algebras-of-smooth-sections} this map is an isomorphism.
 Moreover, by \cite[Theorem~2.86]{Gundogan09Lie-algebras-of-smooth-sections} we
 have that the canonical map
 $C^{\infty}(M,\cent{(\fg)})\to\cent(C^{\infty}(M,\fg))$ is an isomorphism.
 Under the assumptions made we have $\cent(\fg)\cong\R$ (see \cite[Remark
 2.98]{Gundogan09Lie-algebras-of-smooth-sections} or \cite[Remark
 B.4]{NeebWockel07Central-extensions-of-groups-of-sections}) and thus
 $\cent(C^{\infty}(M,\fg))\cong C^{\infty}(M,\R)$. 
 In total, $r$ may thus be
 interpreted as a morphism of Lie algebras
 \begin{equation*}
  r\from \fk\to \out(C^{\infty}(M,\fg))\cong \cV(M).
 \end{equation*}
 Since $\fk$ is finite dimensional the action of   $\fk$ on $C^{\infty}(M,\fg)$ by linear 
 continuous maps induced by $r$ is continuous. 
 Thus we
 have the semi-direct product $C^{\infty}(M,\fg)\rtimes \fk$ giving rise to the
 extension
 \begin{equation}\label{eqn2}
  C^{\infty}(M,\fg)\to C^{\infty}(M,\fg)\rtimes\fk\to\fk.
 \end{equation}
 This extension obviously has $r$ as associated outer action. Since
 $\fz(C^{\infty}(M,\fg))=\{0\}$, the extension theory of topological Lie
 algebras tells us that for each outer action there exists up to equivalence at
 most one extension implementing this outer action \cite[Theorem
 IV.4]{Neeb06Non-abelian-extensions-of-topological-Lie-algebras}. This implies
 that any extension is equivalent to \eqref{eqn2}.
\end{proof}

\begin{proposition}\label{prop:QG_is_no_Lie_group}
 Suppose that $G,K$ are 1-connected finite-dimensional Lie groups and that the
 Lie algebra $\fk:=\Lf(K)$ is semi-simple without $\mf{sl}_{2}(\R)$-summands and
 that $\fg:=\Lf(G)$ is real simple such that $\fg_{\bC}$ is complex simple.
 Then each extension $LG\to \wh{K}\to K$ of Lie groups is equivalent to the
 trivial extension $LG\to LG\times K\to K$.
\end{proposition}

\begin{proof}
 Let $L\fg\to\wh{\fk}\to\fk$ be the associated extension of Lie algebras. By
 Proposition~\ref{prop:extensions_of_current_algebras} this extension is
 uniquely determined by an action $r\from\fk\to\cV(S^{1})$. Clearly, $\im(r)$
 is a finite-dimensinonal subalgebra of $\cV(S^{1})$.

 Let $\bigoplus_{i=i}^{n}\wt{\fk}_{i}$ be a decomposition of $\fk$ into simple
 ideals and set $\fk_{i}:=r(\wt{\fk}_{i})\leq \cV(S^{1})$. Suppose
 $0\neq \fk_{i}\leq \cV(S^{1})$. Then there exists some $x_{i}\in S^{1}$ which
 is not fixed by $\fk_{i}$ (i.e. some vector field contained in $\fk_{i}$ does not vanish
 in $x_{i}$). Then the orbit of the integrated action of the 1-connected Lie
 group $K_{i}$ of $\fk_{i}$ has a one-dimensional orbit and thus the stabiliser
 $K_{x_{i}}\leq K_{i}$ has co-dimension one. Thus $\fk_{i}$ has to have a
 co-dimension one subalgebra $\fk_{x_{i}}:=\Lf(K_{x_{i}})$. Among the simple
 finite-dimensional Lie algebras this is only the case for $\mf{sl}_{2}(\R)$
 \cite[Lemma~2]{Hofmann65Lie-algebras-with-subalgebras-of-co-dimension-one}. It
 follows from this and the assumptions on $\fk$ that $r\from \fk\to \cV(S^{1})$
 has to be trivial. Thus Proposition~\ref{prop:extensions_of_current_algebras}
 implies that $L\fg\to\wh{\fk}\to\fk$ is equivalent to
 $L\fg\to L\fg\times\fk\to\fk$.
 
 In order to now transfer the assertions to the Lie group extension
 $$
 LG\to\wh{K}\to K,
 $$ observe that $\wh{K}$ is a regular Lie group in the sense
 of Milnor since $LG$ and $K$ are \cite[Theorem
 V.1.8]{Neeb06Towards-a-Lie-theory-of-locally-convex-groups}. Thus the Lie
 algebra homomorphisms $\fg\to \wh{\fg}$ and $\wh{\fg}\to L\fg$ that are
 induced by the equivalence $\wh{\fg}\cong L\fg\times \fg$ integrate to Lie
 group homomorphisms since $K$ and $LG$ are 1-connected (and thus is $\wh{K}$).
 This shows that $\wh{K}\cong LG\rtimes K$. Moreover, the triviality of the
 action of $\fk$ on $L\fg$ implies that the action of $K$ on $LG$ has to be
 trivial since $K$ is connected. This finishes the proof.
\end{proof}

\begin{corollary}
 If $G$ is compact and 1-connected, then $QG$ does not admit a Lie group
 structure that coincides on $LG$ with the usual one (cf.\
 \cite{Glockner02Lie-group-structures-on-quotient-groups-and-universal-complexifications-for-infinite-dimensional-Lie-groups,PressleySegal86Loop-groups})
 and admits a morphism $QG\to G$ that turns $LG\hookrightarrow QG\to G$ into an
 extension of Lie groups, whose underlying principal $LG$-bundle is equivalent
 to $\Theta\from QG\to G$. The same is true for the evaluation map
 $\ev_{0,1}\from QG\to G\times G$.
\end{corollary}

\begin{proof}
 It suffices to show that both bundles $\Theta\from  QG\to G$ and
 $\ev_{0,1}\from QG\to G\times G$ are non-trivial. To check this it suffices by
 \cite{MullerWockel07Equivalences-of-Smooth-and-Continuous-Principal-Bundles-with-Infinite-Dimensional-Structure-Group}
 to observe that the underlying topological bundles are non-trivial. If
 $\Theta\from QG\to G$ was trivial, then so would be $QG/G\to G$. But the
 latter is the universal $\Omega G$-bundle, which is not trivial since
 $\Omega G$ is not contractible. To see that $\ev_{0,1}\from QG\to G\times G$
 cannot be trivial one can for instance observe that this would contradict
 $\pi_{n}(QG)\cong \pi_{n}(G)$.
\end{proof}

\begin{remark}\label{rem:circle_group_not_acting_on_pointed_models}
 Note that all the results of this section are very sensitive to the fact that we are dealing with smooth loops and paths (in contrast to continuous
 loops/paths or loops that are smooth except at certain points). 
 This is (at least morally) due to the fact that $\Diff(S^{1})$ (or $S^1$) only acts on smooth loops, and not on loops that have any kind of preferred points.
\end{remark}

\section{Strict Lie 2-group models for the string 2-group from \texorpdfstring{$\wh{LG}$}{LG-hat} do not exist} %
\label{sec:lie_2_group_models_for_lg_}

\begin{tabsection}
 In this short section we will apply the results from the previous section to
 strict models for the string 2-group. We will show that there does not exist
 such a model with source fibre the universal central extension $\wh{LG}$.
 
 Throughout this section, $G$ denotes a simple, compact and 1-connected Lie
 group.
\end{tabsection}

\begin{remark}\label{rem:central-extension-of-the-full-loop-group}
 In this remark we collect the background on the universal central extension
 $\wh{LG}$ of $LG$ that we will need later on in the text. Recall the Kac-Moody
 cocycle
 \begin{equation*}
  \kappa\from  L\fg\times L\fg\to \R,\quad (\mu,\nu)\mapsto \int_{0}^{1}\langle \mu(t),\nu'(t)\rangle dt,
 \end{equation*}
 where $\sprod$ is the Killing form on $\fg$ normalised as discussed at the beginning. This is a continuous cocycle on
 $L\fg$ and a generator of $H^{2}_{c}(L\fg,\R)\cong \R$. From this one
 constructs a central extension $U(1) \to \wh{LG}\to LG$, see for instance
 \cite{Mickelsson85Two-cocycle-of-a-Kac-Moody-group,PressleySegal86Loop-groups,MurrayStevenson01Yet-another-construction-of-the-central-extension-of-the-loop-group,MaierNeeb03Central-extensions-of-current-groups,BaezCransStevensonSchreiber07From-loop-groups-to-2-groups}.
 It is not important how $\wh{LG}$ is constructed exactly, the only thing that
 we will use is that the derived central extension
 $\Lf(\wh{LG})\to \Lf(LG)=L\fg$ is equivalent to the central extension
 $\R \oplus_{\kappa} L\fg\to L\fg$ of Lie algebras.

 What we obtain from $\wh{LG}$ is in particular a lifting bundle gerbe for the
 principal $LG$-bundle $q\from QG\to G$ and the central extension
 $U(1)\to\wh{LG}\to LG$. This gives rise to an action groupoid
 $(\wh{LG}\times QG \toto QG )$, where $\wh{LG}$ acts on $QG$ via the morphism
 $\wh{LG}\to LG$ and the action of $LG$ on $QG$.
\end{remark}

\begin{tabsection}
 For later reference we fix the following definition.
\end{tabsection}

\begin{definition}\label{def:SQG}
 The action groupoid $ (\wh{LG}\times QG \toto QG)$ from the previous remark
 will be denoted by $S(QG)$.
\end{definition}
 We will assume some familiarity with the notion of crossed modules of Lie
 groups. Background on this can be found in
 \cite{BaezLauda04Higher-dimensional-algebra.-V.-2-groups,BaezCransStevensonSchreiber07From-loop-groups-to-2-groups,Neeb07Non-abelian-extensions-of-infinite-dimensional-Lie-groups,Porst08Strict-2-Groups-are-Crossed-Modules,NikolausSachseWockel13A-smooth-model-for-the-string-group,MurrayRobertsStevenson12On-the-existence-of-bibundles}. 
 We will follow the convention to use the conjugation action from the
 \emph{right}. In particular, the notion of a \emph{smooth} crossed module will
 be the one from \cite[Definition
 3.1]{Neeb07Non-abelian-extensions-of-infinite-dimensional-Lie-groups} (with
 the small exception that we consider the right conjugation action). If
 $\cG=(G_{1}\toto G_{0})$ is a strict Lie 2-group, then clearly the
 multiplication functors turn $G_{1}$ and $G_{0}$ into Lie groups. From this we
 obtain a crossed module $\tau\from K\to H$ as follows. We set $K$ to be the
 source fibre $K:=s^{-1}(1)$, which is a closed Lie subgroup since $s$ is a
 submersion. Furthermore, we set $H:=G_{0}$ and $\tau$ to be the restriction of
 the target map $t$ to $K$. Moreover, $H$ acts smoothly on $K$ by conjugation
 with identity morphisms. The requirements on $\cG$ to be a Lie 2-group then
 imply that $\tau\from K\to H$ is a crossed module. For more detail on this see
 in particular
 \cite{BaezLauda04Higher-dimensional-algebra.-V.-2-groups,Porst08Strict-2-Groups-are-Crossed-Modules}.
 Under some moderate further requirements we can further deduce a lifting
 bundle gerbe as follows (see
 \cite{Murray96Bundle-gerbes,Murray10An-introduction-to-bundle-gerbes,SchweigertWaldorf11Gerbes-and-Lie-groups}
 for background on bundle gerbes). In the sequel, $\ul{\pi}_{1}(\cG)\leq G_{1}$ denotes the subgroup of automorphisms of the identity and $\ul{\pi}_{0}(\cG)$ the quotient of isomorphism
 classes in $G_{0}$.  We call $\cG$ \emph{smoothly separable}
 (cf.\ \cite[Definition
 3.6]{NikolausWaldorf11Lifting-Problems-and-Transgression-for-Non-Abelian-Gerbes})
 if $G_{1}/\ul{\pi}_{1}(\cG)$ carries a manifold structure turning $G_{1}\to G_{1}/\ul{\pi}_{1}(\cG)$ into a submersion and the group  $\ul{\pi}_{0}(\cG)$ of isomorphism
 classes carries a Lie group structure turning $G_{0}\to \ul{\pi}_{0}(\cG)$
 into a submersion. In this case the submersion $G_{0}\to \ul{\pi}_{0}(\cG)$
 has as kernel a closed Lie subgroup which is isomorphic to
 $\im(\tau)$ and $G_{0}\to \ul{\pi}_{0}(\cG)$ is a smooth locally
 trivial principal $\im(\tau)$-bundle. Moreover, $\ul{\pi}_{1}(\cG)\to K \to \im(\tau)$ is
 a central extension. This data then gives a \emph{smooth} crossed module of
 Lie groups and, additionally, a $\ul{\pi}_{1}(\cG)$-lifting bundle gerbe over $G$ associated to the principal $\im(\tau)$-bundle and the $\ul{\pi}_{1}(\cG)$-central extension.

\begin{definition}(see
 \cite[Definition 4.10 and Lemma
 4.11]{NikolausSachseWockel13A-smooth-model-for-the-string-group})\label{def:model} Assume that
 $G$ is a simple, compact and 1-connected Lie group. A strict Lie 2-group model
 for the string group (of $G$) is a smoothly separable strict Lie 2-group $\cG$,
 together with isomorphisms $\ul{\pi}_{1}(\cG)\cong U(1)$ and
 $\ul{\pi}_{0}(\cG)\cong G$ such that the underlying bundle gerbe has
 as Dixmier--Douady class a generator of
 $H^{3}(G,\Z)\cong \check{H}^{2}(G,\underline{U(1)})$ with respect to these isomorphisms.
\end{definition}

\begin{tabsection}
 The results of the previous section now read in the context of string group
 models as follows. Recall the concept of topological group cohomology
 \cite{Segal70Cohomology-of-topological-groups,WagemannWockel15A-cocycle-model-for-topological-and-Lie-group-cohomology}
 of which we use the locally smooth model $H^{n}_{\op{loc}}(G,U(1))$ (cf.\
 \cite[Appendix B]{Neeb04Abelian-extensions-of-infinite-dimensional-Lie-groups}
 or \cite[Sections 1 and
 5]{WagemannWockel15A-cocycle-model-for-topological-and-Lie-group-cohomology}).
\end{tabsection}

\begin{proposition}
 Let $G$ be a 1-connected Lie group such that its Lie algebra $\fg$ is real
 simple and that $\fg_{\bC}$ is complex simple. Then each smooth crossed module
 \begin{equation*}
  \tau\from \wh{LG}\to H
 \end{equation*}
 with $\ker(\tau)\cong U(1)$ and $\coker(\tau)\cong G$ has a trivial
 characteristic class in $H^{3}_{\loc}(G,U(1))$. 
  In particular, the underlying
 $U(1)$-lifting bundle gerbe is trivial.
\end{proposition}

\begin{proof}
 From the assumptions it follows that $H$ gives rise to an extension
 \begin{equation*}
  \wh{LG}/U(1)\cong LG\to    H\to G
 \end{equation*}
 of $G$ by $\wh{LG}/U(1)\cong LG$. Since $LG$ and $G$ are regular Lie groups,
 so is $H$ \cite[Theorem
 V.1.8]{Neeb06Towards-a-Lie-theory-of-locally-convex-groups}. From Proposition
 \ref{prop:extensions_of_current_algebras} it follows that the associated Lie
 algebra extension has a Lie homomorphic section. This section integrates to a
 Lie group homomorphism $f\from G\to H$, which is a section of $H\to G$ by the
 uniqueness of the integrating Lie group homomorphism. The characteristic class
 in \cite[Lemma
 3.6]{Neeb07Non-abelian-extensions-of-infinite-dimensional-Lie-groups} is
 constructed from a locally smooth section $\sigma\from G\to H$ and the
 associated map
 \begin{equation*}
  \delta_{\sigma}\from G\times G\to LG,\quad (\gamma,\eta)\mapsto \sigma(\gamma)\sigma(\eta)\sigma(\gamma \eta)^{-1}.
 \end{equation*}
 If we take $\sigma=f$, then $\delta_{\sigma}$ vanishes, and thus all maps that
 are subsequently constructed also can be chosen to vanish. By \cite[Lemma
 3.6]{Neeb07Non-abelian-extensions-of-infinite-dimensional-Lie-groups} the
 characteristic class is independent of these choices and thus it vanishes.
 
 Since the Dixmier-Douady class of a lifting bundle gerbe with trivial
 principal bundle is also trivial, the last assertion follows from the
 existence of a continuous section $f\from G\to H$. 
\end{proof}

\begin{remark}
 The same proof as above shows that each smooth crossed module
 \begin{equation*}
  \tau\from  \wh{C^{\infty}(M,G)}\to H
 \end{equation*}
 with $\coker(\tau)\cong G$ has a trivial characteristic class in
 $H^{3}_{\loc}(G,Z(\wh{C^{\infty}(M,G)}))$. Moreover, the Dixmier-Douady class
 of the bundle gerbe associated to a crossed module with trivial
 characteristic class also vanishes by \cite[Proposition
 5.8]{WagemannWockel15A-cocycle-model-for-topological-and-Lie-group-cohomology}
 and \cite[Problem
 8.1(b)]{NeebWagemannWockel13Making-lifting-obstructions-explicit}.
\end{remark}

\begin{corollary}
 If $G$ is a compact, simple and 1-connected Lie group, then there does not
 exist a strict Lie 2-group model for the string group of $G$ that has $\wh{LG}$ as
 source-fibre.
\end{corollary}

\section{A coherent Lie 2-group model from \texorpdfstring{$\wh{LG}$}{LG-hat}} %
\label{sec:a_coherent_lie_2_group_model_for_lg_}

\begin{tabsection}
 In this section we will show that there exists the next best thing to a strict
 Lie 2-group model for the string group, namely a \emph{coherent smooth}
 2-group model. 
 Recall from \cite[Definition
 19]{BaezLauda04Higher-dimensional-algebra.-V.-2-groups} that a coherent smooth
 2-group is a coherent 2-group in the category of smooth manifolds, i.e., a Lie
 groupoid, together with a smooth multiplication functor and smooth associator,
 left and right unit laws and unit and counit.
 Let us define a \emph{coherent Lie 2-group model} for the string group of $G$ analogously to Definition \ref{def:model}, replacing \emph{strict} by \emph{coherent}.

 We will construct this model by modifying the construction from
 \cite{BaezCransStevensonSchreiber07From-loop-groups-to-2-groups} slightly to
 the setting of smooth loops. This model will be a Lie subgroupoid  of $S(QG)$ and
 the main result of this section will be that the inclusion of this model into
 $S(QG)$ is an adjoint equivalence of Lie groupoids. From this one then then
 constructs a coherent Lie 2-group by transferring all the structure from this
 model to $S(QG)$.
 
 Unless mentioned otherwise, $G$ may be an arbitrary (possible
 infinite-dimensional) Lie group throughout this section.
\end{tabsection}

\begin{definition}[cf. \cite{CarMic}]
 To cure the problems that occur when multiplying two quasi-periodic paths we
 set
 \begin{equation*}
  Q_{\flat}G:=\{\gamma\in QG\mid \gamma^{(n)}(0)=\gamma^{(n)}(1)=0\text{ for all }n\geq 1 \},
 \end{equation*}
 where $\gamma^{(n)}$ denotes the $n$-th derivative of $\gamma$. For
 $\gamma,\eta\in Q_{\flat}G$ we set
 \begin{equation}\label{eqn:multiplication-of-quasi-periodic-paths}
  \gamma\cdot \eta :=(\left.\gamma\right|_{[0,1]}\cdot \left.\eta\right|_{[0,1]})\,\widetilde{\phantom{x}}, 
 \end{equation}
 where $\wt{\beta}$ denotes the extension of the smooth path
 $\beta\colon [0,1]\to G$ with vanishing derivatives at the end-points to a
 quasi-periodic path. This turns $Q_{\flat}G$ into a group, since a
 quasi-periodic path is uniquely determined by its restriction to $[0,1]$,
 where the axioms of a group are directly verified. Moreover, we set
 $L_{\flat}G:=Q_{\flat}G\cap LG$ and
 $\Omega_{\flat}G:=Q_{\flat}G\cap \Omega G$.  
\end{definition}

\begin{remark}\label{rem:submanifolds_of_mappings_with_vanishing_derivative}
 We will need to have various spaces of smooth maps with vanishing derivatives
 endowed with smooth structures. To this end, recall that if
 $\varphi\from U\se G\to \varphi(U)\se \fg$ is a chart of $G$ with
 $\varphi(e)=0$, then for each manifold $M$ (possibly with corners), the map
 \begin{equation}\label{eqn:chart_for_partly_flat_mappings}
  C^{\infty}(M,G) \supseteq C^{\infty}(M,U) \xrightarrow{\varphi_{*}} C^{\infty}(M,\varphi(U))\se C^{\infty}(M,\fg),\quad
  \gamma\mapsto \varphi \circ \gamma
 \end{equation}
 is a chart that determines uniquely a Lie group structure on $C^{\infty}(M,G)$
 \cite{Glockner02Lie-group-structures-on-quotient-groups-and-universal-complexifications-for-infinite-dimensional-Lie-groups}.
 Clearly, this chart preserves the property of maps to have vanishing
 derivative in certain fixed points and thus induced submanifold charts for
 those subspaces. For instance, $L_{\flat}G$ and $\Omega_{\flat}G$ are then
 submanifolds of $LG$ and $\Omega G$ with modelling spaces
 \[
  L_{\flat}\fg:=\{\gamma\in C^{\infty}(\R,\fg)\mid \gamma(t+1)=\gamma(t)\text{ and }\gamma^{(n)}(0)=\gamma^{(n)}(1)=0, \,
      \forall t\in\R,\,n\geq 1 \}
\]
  \noindent and 
  \[
 \Omega_{\flat}\fg:=\{\gamma\in L_{\flat}\fg\mid \gamma(0)=\gamma(1)=0 \},
 \]
 respectively. In particular, we have Lie subgroups $\Omega_{\flat}G  < \Omega G$ and $L_{\flat}G < LG$.
 
 Define $P_{\flat}G:=\{\gamma\in C^{\infty}([0,1],G)\mid \gamma^{(n)}(0)=\gamma^{(n)}(1)=0 \text{ for all }n\geq 1 \}$. Clearly, the map
 \begin{align}\label{eqn5}
  \begin{split}
  Q_{\flat}G &\to P_{\flat}G\\
    \gamma &\mapsto \left.\gamma\right|_{[0,1]}\end{split}
 \end{align}
 is a bijection (of sets) with inverse $\gamma\mapsto \wt{\gamma}$. Now
 $P_{\flat}G$ is a submanifold of $PG:=C^{\infty}([0,1],G)$, and thus it is in
 particular a Lie subgroup. We thus can use \eqref{eqn5} to endow $Q_{\flat}G$
 with a manifold structure that turns it into a Lie group for the group
 multiplication \eqref{eqn:multiplication-of-quasi-periodic-paths}.
\end{remark}

\begin{lemma}
 The following assertions hold.
 \begin{enumerate}
  \item The subset $Q_{\flat}G\se QG$ is a reduction of the principal
        $LG$-bundle $\Theta\from QG\to G$ to an $L_{\flat}G$-bundle
        $\Theta_{\flat}\from Q_{\flat}G\to G$. In particular, $Q_{\flat}G$ is a
        submanifold of $QG$. 
  \item The smooth structure defined on $Q_{\flat}G$ in Remark
        \ref{rem:submanifolds_of_mappings_with_vanishing_derivative} using
        (\ref{eqn5}) and the smooth structure induced on $Q_{\flat}G$ from $QG$
        coincide. In particular, it also is a Lie group with respect to the
        smooth structure induced from $QG$.
 \end{enumerate}
\end{lemma}

\begin{proof}
 Since the sections $gU\to QG$, $x\mapsto \gamma_{x}$ constructed in Remark
 \ref{rem:QG} actually take values in $Q_{\flat}G$ and since two elements in
 $Q_{\flat}G$ in the same fibre of $\Theta$ differ by an element of $L_{\flat}G$ the
 first assertion is immediate.
 
 To check the second assertion, we first note that the map
 \begin{equation*}
  \{\gamma\in P_{\flat}G\mid \gamma(0)=\gamma(1) \}\to L_{\flat}G,\quad \gamma\mapsto \wt{\gamma}
 \end{equation*}
 is a diffeomorphism, since its coordinate representation in the chart
 $\varphi_{*}$ in (\ref{eqn:chart_for_partly_flat_mappings}) is a linear homeomorphism. If we also set
 $\Theta_{\flat}(\gamma):=\gamma(1)\cdot \gamma(0)^{-1}$ for $\gamma\in P_{\flat}G$, then
 $\Theta_{\flat}\from P_{\flat}G\to G$ is smooth since the evaluation maps are so.
 Consequently, the map 
 \begin{align}\label{eqn6}
  \{\gamma\in P_{\flat}G\mid \Theta(\gamma)\in gU \}& \xrightarrow{\gamma\mapsto \wt{\gamma}} \{\gamma\in Q_{\flat}G\mid \Theta(\gamma)\in gU \} \xrightarrow{\Phi_{F_{g}}} gU\times L_{\flat}G,\nonumber\\ \gamma & \mapsto (\Theta(\gamma), \gamma^{-1}_{\Theta(\gamma)}\cdot \wt{\gamma})
 \end{align}
 is smooth since
 $\gamma^{-1}_{\Theta(\gamma)}\cdot \wt{\gamma}=(\left.\gamma^{-1}_{\Theta(\gamma)}\right|_{[0,1]}\cdot {\gamma})\,\widetilde{\phantom{x}}$.  
 Likewise, the inverse
 $(x,\gamma)\mapsto \left.\gamma_{x}\right|_{[0,1]}\cdot \left.\gamma\right|_{[0,1]}$
 is smooth and thus \eqref{eqn6} is a diffeomorphism. This implies that the
 smooth structures coincide.
\end{proof}

 As in Remark \ref{rem:QG} we obtain from the $L_{\flat}G$-bundle
 $\Theta_{\flat}\from Q_{\flat}G\to G$ and the equivariant diffeomorphism
 $\Psi_{\flat}\from L_{\flat}G/G\xrightarrow{\cong} \Omega_{\flat}G$ the locally trivial
 $\Omega_{\flat}G$-bundle $Q_{\flat}G/G\to G$. The subgroup
 \[
  Q_{\flat,*}G:=\{\gamma\in Q_{\flat}G\mid \gamma(0)=e \}
\]
is a Lie subgroup of
 $Q_{\flat}G$ and the kernel of $\ev_{1}\from Q_{\flat,*}G\to G$ is $\Omega_{\flat}G$. Then
 $\Psi_{\flat}$ extends to a diffeomorphsim
 \begin{equation*}
  \Psi_{\flat}\from   Q_{\flat}G/G\to Q_{\flat,*}G,\quad 
  [\gamma]\mapsto \gamma\cdot \gamma(0)^{-1}.
 \end{equation*}
 This makes the diagram
 \begin{equation*}
  \vcenter{  \xymatrix{
  L_{\flat}G/G	 \xyhookrightarrow{1.8em}  \ar[d]^{\left.{\Psi_{\flat}}\right|_{L_{\flat}G/G}} & Q_{\flat}G/G\ar[d]^{\Psi_{\flat}} \ar[r]^{\Theta_{\flat}} & G\ar@{=}[d]& \\
  \Omega_{\flat} G \xyhookrightarrow{2.8em} & Q_{\flat,*}G\ar[r]^{\ev_{1}} & G.
  }}
 \end{equation*}
 commute and $\Psi_{\flat}$ is an isomorphism of locally trivial bundles. Note that
 \begin{equation*}
  \Omega_{\flat}G\hookrightarrow   Q_{\flat,*}G\xrightarrow{\ev_{1}} G
 \end{equation*}
 is in fact an extension of Lie groups, i.e., a sequence of Lie groups such
 that $Q_{\flat,*}G$  is a principal $\Omega_{\flat}G$-bundle over $G$. This is our model
 of the smooth path-loop fibration that we will use in the sequel.
 
 We now show that $Q_{\flat,*}G$ is (weakly) contractible. In fact, the restriction
 $\gamma\mapsto r(\gamma):=\left.\gamma\right|_{[0,1]}$ defines a homotopy
 equivalence $r\from \Omega G\to \Omega_{c}G$ and
 $\Omega_{\flat}G\hookrightarrow \Omega G$ is a homotopy equivalence as well (a
 homotopy inverse is for instance given by reparametrisation so that loops have
 vanishing derivatives at one point). Thus $r|_{\Omega_{\flat}G}\from \Omega_{\flat}G\to \Omega_{c}G$
 is a homotopy equivalence. From the Five Lemma, applied to the long exact
 sequence for homotopy groups of a fibration, it follows that
 \begin{equation*}
  r\from Q_{\flat,*}G\to C_{*}([0,1],G)
 \end{equation*}
 is a weak homotopy equivalence. 
 If $G$ is metrisable, then both manifolds
 are metrisable, and it follows from
 \cite{Palais66Homotopy-theory-of-infinite-dimensional-manifolds} that
 $Q_{\flat,*}G$ is contractible.

\begin{remark}
 Now suppose that $G$ is simple, compact and 1-connected. The central extension
 $\wh{LG}\to LG$ restricts to a central extension of the Lie subgroup
 $\Omega_{\flat}G\se LG$ which we denote by $\wh{\Omega_{\flat} G}$. The cocycle of the
 corresponding restricted Lie algebra extension is the restriction of the
 Kac-Moody cocycle to
 \begin{equation*}
  \kappa_{0}\from  \Omega_{\flat}\fg\times \Omega_{\flat}\fg\to \R,\quad (\mu,\nu)\mapsto \int_{0}^{1}\langle \mu(t),\nu'(t)\rangle dt.
 \end{equation*}
If $x.\gamma$ denotes the pointwise adjoint action
 of $Q_{\flat,*}G$ on $\Omega_{\flat}\fg$ from the right  we have
 $\kappa (x.\gamma,y.\gamma)=\kappa(x,y) + \alpha(\gamma, [x, y])$ 
 where 
  $$
 \alpha(\gamma, z) = \int_0^1 \langle z(t), \gamma'(t) \gamma(t)^{-1} \rangle dt 
 $$
for all $\gamma\in Q_{\flat,*}G$ and all
 $x,y, z \in \Omega_{\flat}\fg$. 
  Thus the right conjugation
 action of $Q_{\flat,*}G$ on $\Omega_{\flat}G$ lifts by \cite[Theorem
 V.9]{MaierNeeb03Central-extensions-of-current-groups}  to a \emph{unique} right action
 \begin{equation*}
  \lambda\from \wh{\Omega_{\flat}G}\times Q_{\flat,*}G\to \wh{\Omega_{\flat}G},\quad (\wh{\eta},\gamma)\mapsto \wh{\eta}.\gamma
 \end{equation*}
 such that $Q_{\flat,*}G$ acts trivially on the central
 $U(1)\se \wh{\Omega_{\flat}G}$. From the construction it follows that the
 composition
 $\tau \from \wh{\Omega_{\flat}G}\to \Omega_{\flat}G\hookrightarrow Q_{\flat,*}G$,
 together with the action $\lambda$ defines a smooth crossed module. This gives
 rise to a strict Lie 2-group
 $( \wh{\Omega_{\flat}G}\rtimes Q_{\flat,*}G  \toto Q_{\flat,*}G)$ with source map
 $s(\wh{\eta},\gamma)=\gamma$, target map $t(\wh{\eta},\gamma)=\eta\cdot\gamma$
 and composition map
 $m((\wh{\eta},\gamma),(\wh{\eta}',\gamma'))=(\wh{\eta}\cdot\wh{\eta}',\gamma')$
 (due originally to \cite{BrownSpencer76}, but see \cite[Example
 4.3]{NikolausSachseWockel13A-smooth-model-for-the-string-group}, \cite[Section
 5]{Porst08Strict-2-Groups-are-Crossed-Modules} or \cite[Proposition
 32]{BaezLauda04Higher-dimensional-algebra.-V.-2-groups} for recent treatments applicable to our context.). Note that
 this Lie 2-group is a smooth version of the Lie 2-group that has been
 constructed in
 \cite{BaezCransStevensonSchreiber07From-loop-groups-to-2-groups}, in the sense that it uses loops that are smooth maps $S^1\to G$.
\end{remark}

\begin{definition}\label{def:QstarflatQG}
 The strict Lie 2-group $( \wh{\Omega_{\flat}G}\rtimes Q_{\flat,*}G  \toto Q_{\flat,*}G)$
 from the preceding remark will be denoted by $S(Q_{\flat,*}G)$ and the smooth
 functor 
 $$
 S(Q_{\flat,*}G) \to S(QG)
 $$
  of Lie groupoids, given by the inclusions
 \begin{equation*}
  Q_{\flat,*}G\hookrightarrow QG\quad\text{ and }\quad \wh{\Omega_{\flat}G}\times Q_{\flat,*}G\hookrightarrow \wh{LG}\times QG,
 \end{equation*}
 will be denoted by $\iota\from S(Q_{\flat,*}G) \to S(QG)$.
\end{definition}

Recall that a weak equivalence of Lie groupoids $F\from X\to Y$ is a functor that is:

\begin{itemize}
 \item smoothly fully faithful, in the sense that
       \begin{equation*}
        \xymatrix{
        X_1 \ar[r]^{F_1} \ar[d]_{(s,t)} & Y_1 \ar[d]^{(s,t)} \\
        X_0\times X_0 \ar[r]_-{F_0\times F_0} & Y_1\times Y_1
        }
       \end{equation*}
       is a pullback, and:
 \item smoothly essentially surjective, in the sense that the map
       \begin{equation*}
        X_0\times_{Y_0,s} Y_1 \xrightarrow{\pr_2} Y_1 \xrightarrow{t} Y_0
       \end{equation*}
       is a surjective submersion.
\end{itemize}
While not our final and strongest result, the following theorem is sufficient to give $S(QG)$ the structure of a so-called stacky Lie 2-group (see for instance \cite{Schommer-Pries10Central-Extensions-of-Smooth-2-Groups-and-a-Finite-Dimensional-String-2-Group,Blohmann08Stacky-Lie-groups}).

\begin{Theorem}\label{thm:iota_is_weak_equiv}
Let $G$ be a simple, compact and 1-connected Lie group. The functor $\iota$ is a weak equivalence of Lie groupoids.
\end{Theorem}

\begin{proof}
 We first need to show that
 \begin{equation}\label{eqn:smooth-fully-faithful}
  \vcenter{  \xymatrix{
  \wh{\Omega_{\flat} G}\times Q_{\flat,*}G\ar[d]_{(s,t)} \ar[r]^{\iota_{1}} & \wh{LG}\times QG\ar[d]^{(s,t)}\\
  Q_{\flat,*}G\times Q_{\flat,*}G \ar[r]_-{\iota_{0}\times \iota_{0}} & QG\times QG
  }}
 \end{equation}
 is a pullback. But since $Q_{\flat,*}G$ is a subspace of $QG$ and
 $\wh{\Omega_{\flat}G}$ is the pullback of $\wh{LG}$ along
 $\Omega_{\flat}G \nbinto LG$ we have that \eqref{eqn:smooth-fully-faithful} is
 clearly a pullback.
 
 Secondly, we need to show that
 \begin{equation}\label{eqn:smooth-essential-surjective}
  Q_{\flat,*}G\times_{QG}  (\wh{LG}\times QG)\simeq \wh{LG}\times Q_{\flat,*}G  \to LG\times Q_{\flat,*}G \to QG
 \end{equation}
 is a surjective submersion. The right-most map in
 \eqref{eqn:smooth-essential-surjective} is the `action' of $LG$ on
 $Q_{\flat,*}G$, considered as a sub-bundle of $QG$. Since $\wh{LG} \to LG$ is
 a principal bundle, it is a surjective submersion, so we are reduced to
 showing the map $LG\times Q_{\flat,*}G \to QG$ is a surjective submersion.
 This is a case of the following setup, whose application to the case $H = LG$,
 $K = \Omega_{\flat}G$, $P = QG$ and $P_K = Q_{\flat,*}G$ finishes the proof.
\end{proof}
 
 \begin{lemma}
  Let $H$ be a Lie group and $K\leq H$ be a Lie subgroup. If $P\to M$ is a
  principal $H$-bundle that reduces to a principal $K$-bundle $P_{K}\to M$,
  then the map $a\from H\times P_K \to P$ (the $H$-action on $P$ restricted to
  $P_K$) is a surjective submersion.
 \end{lemma}
 
 \begin{proof}
  We can reduce immediately to the case that $P$ (and hence $P_K$) is a trivial
  bundle, for a map is a surjective submersion if it is one locally on the
  codomain. We can without loss of generality consider $M=*$, and so $a$ is
  just
  \begin{equation*}
   H\times K \to H, \quad (h,k)\mapsto h\cdot k.
  \end{equation*}
  Since this maps factors through the diffeomorphism $H\times K\to H\times K$,
  $(h,k)\mapsto (h\cdot k,k)$ and the projection $\pr_{1}\from H\times K\to H$,
  it clearly is a submersion.
 \end{proof}

Weak equivalences of Lie groupoids give rise to equivalences of associated differentiable stacks on the site of manifolds in which the Lie groupoids live. 
Since forming the associated differentiable stack of a Lie groupoid preserves finite products, the strict 2-group structure on $S(Q_{\flat,*}G)$ gives a strict group stack structure on the associated stack. 
It is this 2-group structure that transfers across the equivalence of stacks to give the structure of a differentiable group stack\footnote{this is a structure equivalent to   what is meant by the phrase `stacky Lie 2-group'} on the stack associated to $S(QG)$.

However, there is little control over this group stack structure, so we would like to rigidify the situation, in the sense that multiplication will be given by a smooth functor, rather than a map of associated stacks (equivalently: by an anafunctor, or by a principal bibundle). 
We can even choose data, up to a known ambiguity, that will give an explicit multiplication functor, together with the rest of the coherent 2-group structure.

We shall do this by showing $S(QG)$ and $S(Q_{\flat,*}G)$ are not only weakly equivalent, but equivalent in the strong sense of there being an adjoint equivalence between them in the 2-category of Lie groupoids. 
Recall that having such an adjoint equivalence means there is a smooth functor $\iota\colon S(Q_{\flat,*}G)\to S(QG)$, a smooth functor $\rho\colon S(QG)\to S(Q_{\flat,*}G)$ and smooth natural isomorphisms $\varepsilon\colon \iota\circ \rho\Rightarrow \id_{S(QG)} $ and $\eta\colon \id_{S(Q_{\flat,*}G)} \Rightarrow \rho\circ\iota$.
These need to satisfy the triangle identities, which in this instance reduce to $\rho(\varepsilon_f) = \eta_{\rho(f)}^{-1}$ and $\iota(\eta_g) = \varepsilon_{\iota(g)}^{-1}$.
Further, since $\iota$ and $\rho$ are both fully faithful and essentially surjective, each triangle identity determines the other, and further, each of $\varepsilon$ and $\rho$ determines the other uniquely.

\begin{Theorem}\label{thorem:smooth_adj_equiv}
 Let $G$ be a simple, compact and 1-connected Lie group. Then the smooth functor
 \begin{equation*}
  \iota\from S(Q_{\flat,*}G) \to S(QG)
 \end{equation*}
 is part of a smooth adjoint equivalence of Lie groupoids.
\end{Theorem}

\begin{proof}
 We will show \eqref{eqn:smooth-essential-surjective} has a smooth section.
 From this the claim follows, since the standard construction of a left adjoint
 quasi-inverse
 \begin{equation*}
  \rho \from S(QG)\to S(Q_{\flat,*}G)
 \end{equation*}
 for the fully faithful and essentially surjective functor $\iota$ (see for
 instance \cite[Proposition~3.4.3]{Borceux94Handbook-of-categorical-algebra.-1}
 or \cite[Theorem~IV.4.1]{Mac-Lane98Categories-for-the-working-mathematician})
 then carry over verbatim to the smooth case. In fact, the choices of objects
 and morphisms that enter this construction of $\rho$ amount exactly to the
 choice of a section of
 \begin{equation} \label{eqn3}
  Q_{\flat,*}G\times_{QG}  (\wh{LG}\times QG)\to QG
 \end{equation}
 from \eqref{eqn:smooth-essential-surjective}. Moreover, the argument that
 there exist unique morphisms yields, together with the pull-back property of
 \eqref{eqn:smooth-fully-faithful}, a smooth functor. So in case this section
 is smooth the standard construction yields smooth functors and natural
 transformations.
 
 To construct a smooth section of \eqref{eqn3}, we choose
 $F\from [0,1]\times\R\to\R$ smooth such that $F_{0}=\id_{\R}$, that each
 $F_{t}\from \R\to \R$ satisfies $F_{t}(0)=0$ and $F_{t}(x+n)=F_{t}(x)+n$ for
 all $n\in \Z$ and such that $F_{1}\from \R\to \R$ has vanishing derivatives of
 all orders on $\Z$ (i.e., satisfies $F_{1}^{(n)}(k)=0$ for all $n\geq 1$ and
 $k\in\Z$). Such a $F$ can be obtained by first constructing $\varphi$ as the
 $\Z$-equivariant extension of $\left.\wt{\varphi}\right|_{[0,1]}$, where
 $\wt{\varphi}\from \R\to \R$ is smooth with
 $\left.\wt{\varphi}\right|_{(-\infty,0]}\equiv 0$,
 $\left.\wt{\varphi}\right|_{[1,\infty)}\equiv 1$. Such a function $\varphi$
 will be called a \emph{smooth staircase function}. With this, we can then set
 \begin{equation*}
  F_t(x):=(1-t)\cdot x+t\cdot \varphi(x).
 \end{equation*}
 For $F$ constructed this way we then have $\varphi=F_{1}$, and if $F$ is an
 arbitrary function with the above properties we define $\varphi:= F_{1}$. From
 this we obtain the smooth map $G_{F}\from QG\to P_*LG$  with
 \begin{equation*}
  G_{F}(\gamma)(s):=\gamma\cdot (\gamma \circ F_{s})^{-1}.
 \end{equation*}
 For each value of $s$, $\gamma \circ F_{s}$ is a reparametrisation of $\gamma$
 that leaves the base-point fixed and $\gamma \circ F_{0}=\gamma$. Thus
 $\gamma\cdot( \gamma \circ F_{s})^{-1}\in \Omega G\se LG$ and
 $\gamma\cdot( \gamma \circ F_{0})^{-1}\equiv e$, and $G_{F}(\gamma)$ is an
 element of $P_{*}L G$.
 
 In order to lift $G_{F}$ to $\wh{LG}$ we observe that $P_*LG$  
 is contractible, so that the pullback of the central extension $\wh{LG}\to LG$
 along the evaluation map $\ev_{1}\from P_{*}LG\to LG$ (cf.\ Remark
 \ref{rem:central-extension-of-the-full-loop-group}) has a smooth global
 section \cite[Section
 8]{Neeb02Central-extensions-of-infinite-dimensional-Lie-groups}. This gives
 rise to the diagram
 \begin{equation*}
  \xymatrix{
  \ev_{1}^{*}(\wh{LG})\ar[r]^(.6){\pi_{2}}\ar[d]_{\pi_{1}} & \wh{LG}\ar[d]^{q}
  \\
  P_{*}LG\ar[r]^{\ev_{1}}\ar@/_{1em}/[u]_{\sigma}&LG
  }
 \end{equation*}
 With this we now define the smooth map
 \begin{equation*}
  \tau\from QG\to \wh{LG}\times Q_{\flat,*}G ,\quad \gamma\mapsto ((\pi_{2}(\sigma(G_{F}(\gamma)))), (\gamma \circ \varphi)(0)^{-1}\cdot \gamma \circ \varphi).
 \end{equation*}
 In order to check that $\tau$ is a section of \eqref{eqn3} we observe that
 \begin{equation*}
  (q \circ \pi_{2} \circ \sigma \circ G_{F})(\gamma)= (\ev_{1}\circ \pi_{1} \circ \sigma \circ G_{F})(\gamma)=(\ev_{1} \circ G_{F})(\gamma)= \gamma \cdot (\gamma \circ \varphi)^{-1}\cdot (\gamma \circ \varphi)(0)
 \end{equation*}
 implies
 \begin{equation*}
  t(\pr_{1}( \tau(\gamma)))=\gamma \cdot (\gamma \circ \varphi)^{-1}\cdot(\gamma \circ \varphi)(0)\cdot(\gamma \circ \varphi)(0)^{-1}\cdot (\gamma \circ \varphi)=\gamma.
 \end{equation*}
 This finishes the proof.
\end{proof}

\begin{remark}
 A comment on the choice made in the previous proof. The only choice that
 enters the construction of the adjoint equivalence is the section
 \begin{equation*}
  \tau\from QG\to \wh{LG}\times Q_{\flat,*}G
 \end{equation*}
 of the map
 \begin{equation}\label{eqn7}
    \wh{LG}\times Q_{\flat,*}G  \to QG
 \end{equation}
from \eqref{eqn:smooth-essential-surjective}.
 The remaining structure is fixed by this, and in fact any two choices of sections give rise to uniquely isomorphic quasi-inverses to $\iota$. This means that there is a contractible groupoid worth of choices of sections.

 It is unclear to the authors what a reasonable topology on the spaces of sections of
 \eqref{eqn7} is, so it does not make sense at
 this point to say that this space of sections is contractible (as opposed to the \emph{groupoid} of sections). 
 However, what entered in
 the construction of $\tau$ in the previous proof was just the smooth map
 \begin{equation*}
  F\from [0,1]\times \R\to \R
 \end{equation*}
 such that $F_{0}=\id_{\R}$, that each $F_{t}\from \R\to \R$ satisfies
 $F_{t}(0)=0$ and $F_{t}(x+n)=F_{t}(x)+n$ for all $n\in \Z$ and such that
 $F_{1}\from \R\to \R$ has vanishing derivatives of all orders on the subset $\Z \subset \R$. This
 clearly is an affine space modelled on the vector space
 \[
  \{f\in C^{\infty}([0,1]\times S^{1},\R)\mid f_{s}(0)=0,\ f_{0}(t)=0,\ f^{(n)}_{1}(0)=0 \ 
  \text{ for all }s\in [0,1],\ t\in S^{1},n\geq 1 \}.
 \]
 In particular, the choices that enter into this particular natural construction of the
 adjoint equivalence form a contractible space.
\end{remark}

From the smooth equivalence $\iota\from S(Q_{\flat,*}G)\to S(QG)$ one can
construct on $S(QG)$ the structure of a smooth coherent 2-group. In fact the
2-category $\cohLie$ is just the 2-category of coherent group objects in
$\LieGrpd$. In particular, $S(Q_{\flat,*}G)$ is an object of $\cohLie$ (it
even is a strict Lie 2-group). Since $\iota\from S(Q_{\flat,*}G)\to S(QG)$ is
an equivalence in this strict 2-category we may transport all the structure
that is internal to $\LieGrpd$ through this equivalence.

More explicitly, we can define the multiplication in $S(QG)$ as follows. 
Note that the multiplication in $S(Q_{\flat,*}G)$ will be written as plain juxtaposition.
We first choose a quasi-inverse $\rho \from S(QG)\to S(Q_{\flat,*}G)$ as above and a construct the counit $\varepsilon\colon \rho \iota \Rightarrow \id$. 
Having fixed these choices, we set
\[
  f\otimes g = \iota[\rho(f)\rho(g)],\qquad f,g\in QG
\]
and define the monoidal unit $I := \iota(1)$, where $1$ is the monoidal unit in $S(Q_{\flat,*}G)$, namely the constant path at the identity.

From the counit $\varepsilon$, we define the associator $a$ of the 2-group $S(QG)$ as the composite
\begin{multline}\label{eq:associator}
  a_{f,g,h}\colon (f\otimes g)\otimes h = \iota[\rho(\iota[\rho(f)\rho(g)])\rho(h)] 
  \xrightarrow{\iota( \varepsilon_{\rho(f)\rho(g)}\rho(h)) }
  \iota[\rho(f)\rho(g)\rho(h)] \\
  \xrightarrow{\iota( \rho(f) \varepsilon^{-1}_{\rho(g)\rho(h)})}
  \iota[\rho(f)\rho(\iota[\rho(g)\rho(h)])] = f\otimes (g\otimes h).
  \end{multline}
Using both the counit and the unit $\eta$ (which at $f$ is the unique lift through $\rho$ of $\varepsilon_{\rho(f)}^{-1}\colon \rho(f) \to \rho(\iota(\rho(f)))$), the right ($r_f$) and left ($l_g$) unitors are:
\begin{equation}\label{eq:right_unitor}
  r_f \colon f\otimes I = \iota(\rho(f)\rho(\iota(1))) \xrightarrow{\iota(\id_{\rho(f)}\varepsilon_1)} \iota(\rho(f)1) = \iota(\rho(f)) \xrightarrow{\eta_f^{-1}} f
\end{equation}
and
\begin{equation}\label{eq:left_unitor}
  l_g \colon I \otimes g = \iota(\rho(\iota(1))\rho(g)) \xrightarrow{\iota(\varepsilon_1\id_{\rho(g)})}
  \iota(1\rho(g)) = \iota(\rho(g)) \xrightarrow{\eta_g^{-1}} g.
\end{equation}

\begin{lemma}\label{lemma:monoidal_adjoint_equivalence}
  The associator (\ref{eq:associator}), unit $I$ and unitors (\ref{eq:right_unitor}), (\ref{eq:left_unitor}) make $S(QG)$ into a monoidal Lie groupoid.
  Moreover, the adjoint equivalence of Lie groupoids from Theorem~\ref{thorem:smooth_adj_equiv} lifts to an adjoint equivalence between the monoidal Lie groupoids $S(QG)$ and $S(Q_{\flat,*}G)$.
\end{lemma}

\begin{proof}
  The proof that the required monoidal category coherence diagrams for $a$, $l$ and $r$ commute reduces to a routine calculation using repeated application of naturality for $\varepsilon$ and $\eta$.

  Our strategy for proving that $S(QG)$ and $S(Q_{\flat,*}G)$ are equivalent as monoidal Lie groupoids is via the following standard result: an adjunction between the categories underlying monoidal categories will lift to a \emph{lax monoidal} adjunction as soon as one of the functors lifts to a \emph{strong} monoidal functor \cite{Kelly74Doctrinal-adjunction}.
  This argument internalises to Lie groupoids without issue, as structure is being uniquely determined by existing data, with no choices or case analysis.
  We will in fact prove that $\rho$ is a strong monoidal functor for our choice of coherence data for $S(QG)$, and then $\iota$, $\varepsilon$ and $\eta$ will be monoidal, as needed.
  A priori, $\iota$ is only \emph{lax} monoidal, but as $\varepsilon$ and $\eta$ are invertible, the formulas defining the monoidal structure for $\iota$ in \cite[\S 2]{Kelly74Doctrinal-adjunction} give $\iota$ to be strong monoidal, which is what we want.

  Recall that a monoidal functor (see eg \cite[\S 2]{BaezLauda04Higher-dimensional-algebra.-V.-2-groups}) $F\colon C \to D$ requires the following diagram to commute:
  \[
    \xymatrix{
    (F(x)\otimes F(y)) \otimes F(z) \ar[r] \ar[d]_{a^D_{F(x),F(y),F(z)}} & F(x\otimes y) \otimes F(z) \ar[r] & F((x\otimes y)\otimes z) \ar[d]^{F(a^C_{x,y,z})}\\
    F(x) \otimes (F(y) \otimes F(z)) \ar[r] & F(x) \otimes F(y\otimes z) \ar[r] & F(x\otimes (y\otimes z))
    }
  \]
  where the unlabelled (invertible) arrows are coherence data for $F$. 
  Thus if $F$ is fully faithful, $a^C_{x,y,z}$ is determined by $a^D_{F(x),F(y),F(z)}$ and the coherence structure for $F$.
  We shall consider the case where $F=\rho\colon S(QG) \to S(Q_{\flat,*}G)$, and of course as $S(Q_{\flat,*}G)$ is strictly monoidal, its associator is the identity.
  Thus the associator (\ref{eq:associator}) for $S(QG)$ is determined uniquely by the choice of coherence structure for $\rho$, and this is in fact our definition of $a$ above, for a particular choice of isomorphism $\rho(f)\rho(g) \to \rho(f\otimes g)$, namely 
  \[
    \rho(f)\rho(g) \xrightarrow{\varepsilon^{-1}_{\rho(f)\rho(g)}} \rho(\iota(\rho(f)\rho(g)))  = \rho(f\otimes g).
  \]
  Additionally, unitor coherence for $\rho$ will determine the left and right unitors of $S(QG)$ uniquely (again, by full-faithfulness of $\rho$), and our unitors (\ref{eq:right_unitor}) and (\ref{eq:left_unitor}) were defined via the putative unitor coherence diagrams for $\rho$. A similar argument holds for the unit $I$ and the unit coherence for $\rho$.

  Thus, by virtue of the definitions of the coherence data for $S(QG)$, the functor $\rho$ admits coherence data making it a monoidal functor $S(QG) \to S(Q_{\flat,*}G)$ between monoidal groupoids.
  Then $(\rho,\iota)$ is an adjunction between monoidal groupoids, but $\eta$ and $\varepsilon$ are invertible, so it is a monoidal adjoint equivalence.
\end{proof}
 
Since arrows and 2-arrows in the 2-category of coherent 2-groups are defined to be monoidal functors and natural transformations between the underlying monoidal groupoids \cite[Definitions 3.2, 3.3]{BaezLauda04Higher-dimensional-algebra.-V.-2-groups}, once we show $S(QG)$ has a coherent weak inversion functor, then we have shown that it is equivalent to $S(Q_{\flat,*}G)$ as a coherent 2-group.

Now define the coherent inversion functor of $S(QG)$, using the strict inversion functor $(-)^{-1}$ of $S(Q_{\flat,*}G)$, as:
\[
  \overline{\;\cdot\;}\colon S(QG) \xrightarrow{\rho} S(Q_{\flat,*}G) \xrightarrow{(-)^{-1}} S(Q_{\flat,*}G)  \xrightarrow{\iota} S(QG)
 \]
and the left and right inverters to be:
\begin{equation}\label{eq:right_inverter}
  f\otimes \overline{f} = \iota(\rho(f)\rho(\iota(\rho(f)^{-1}))) \xrightarrow{\id_{\rho(f)}\varepsilon_{\rho(f)^{-1}}} \iota(\rho(f)\rho(f)^{-1})  = \iota(1) = I
  \end{equation}
and
\begin{equation}\label{eq:left_inverter}
  \overline{g} \otimes g = \iota(\rho(\iota(\rho(g)^{-1}))\rho(g)) \xrightarrow{\varepsilon_{\rho(g)^{-1}}\id_{\rho(g)}} \iota(\rho(g)^{-1}\rho(g))  = \iota(1) = I.
\end{equation}
Repeated use of naturality ensures that the inverters are coherent.
Since the equivalence of the Lie groupoids (over $G$) gives equality of the Dixmier--Douady classes of the associated bundle gerbes, the class of $S(QG)$ is a generator of $H^3(G,\mathbb{Z})$. 
Hence we have proved:

\begin{Theorem}\label{thm:coherent_2-group_structure}
  The Lie groupoid $S(QG)$ together the the coherence maps (\ref{eq:associator}), (\ref{eq:right_unitor}), (\ref{eq:left_unitor}), (\ref{eq:right_inverter}) and (\ref{eq:left_inverter}) is a coherent Lie 2-group, equipped with an adjoint equivalence to the strict Lie 2-group $S(Q_{\flat,*}G)$.
  Hence $S(QG)$ is a coherent model for the string group of $G$. 
\end{Theorem}

 \begin{remark}
If one chooses a section of (\ref{eqn3}) using a function $F$ as in the proof of Theorem~\ref{thorem:smooth_adj_equiv} (say via the construction of an explicit smooth function $\varphi$ given there), then one can write out the coherence maps (\ref{eq:associator}), (\ref{eq:right_unitor}), (\ref{eq:left_unitor}), (\ref{eq:right_inverter}) and (\ref{eq:left_inverter}) explicitly in terms of $F$ (or of $\varphi$).
 \end{remark}

 Even better than   the previous theorem is the result that the 2-groupoid of all possible lifts of the equivalence $\iota$ of Lie groupoids to an equivalence of coherent 2-groups is contractible.
 That is: a contractible 2-groupoid of lits of a choice of coherent Lie 2-group structure on $S(QG)$ plus an adjoint equivalence of coherent Lie 2-groups with $S(Q_{\flat,*}G)$. 
 The approach we used above is \emph{one} way to construct a coherent Lie 2-group structure on $S(QG)$, but in fact \emph{any} way to build such a structure will be equivalent to the one we gave.
 The following argument is due to Chris Schommer-Pries \cite{SchommerPriesMO}.

 We first define the 2-groupoid of lifts. 
 The reader should keep in mind the forgetful 2-functor $\cohLie \to \LieGrpd$ that forgets the coherent 2-group structure. 
 Suppose that $U: \mathcal{A} \to \mathcal{C}$ is a 2-functor between 2-categories. 
 Let $x$ and $y$ be objects of $\mathcal{C}$, and $f\colon x \to y$ a fixed morphism. Suppose that we have a lift $Y \in \mathcal{A}$ of $y \in \mathcal{C}$, that is an object such that $U(Y) = y$ (in our case, $Y$ is $S(Q_{\flat,*}G)$, $y$ is the underlying Lie groupoid and $f$ is $\iota$). 

 We can define the 2-category $\Lift(f; Y)$ of lifts of $f$ relative to $Y$. It is the 2-category whose objects are arrows $F\colon X \to Y$ such that $U(F) = f$, whose morphisms are the obvious triangles (with 2-isomorphisms witnessing the commutativity of the triangle).
 The top edge of the triangle maps via $U$ to the identity morphism of $x$. 
 The 2-morphisms are maps of the top edges of triangles such that the obvious diagram commutes and such that the 2-morphisms map via $U$ to the identity 2-morphism of the identity 1-morphism of $x$. 

 \begin{lemma}[Schommer-Pries, \cite{SchommerPriesMO}]
 Let $U\colon \mathcal{A} \to \mathcal{C}$ be a 2-functor, and assume:
 \begin{enumerate}
 
  \item The 2-functor $U$ is \emph{conservative} in the following sense: a 1-morphism of $\mathcal{A}$ is an equivalence in $\mathcal{A}$ if and only if its image under $U$ is an equivalence in $\mathcal{C}$; a 2-morphism of $\mathcal{A}$ is an isomorphism if and only if its image under $U$ is an isomorphism; and two 2-morphisms of $\mathcal{A}$ are equal if and only if they are equal after applying $U$.

  \item One  can lift structure along equivalences and invertible transformations. 
  Meaning that if $f:x \to y$ is in $\mathcal{C}$ and $Y \in \mathcal{A}$ lifts $y$, then there exists \emph{some} lift of $f$ to $\mathcal{A}$. 
  Moreover if $\alpha\colon f \cong g$ is a 2-isomorphism in $\mathcal{C}$ and one can lift $f$ and $g$ to parallel arrows in $\mathcal{A}$, then there exists \emph{some} compatible lift of $\alpha$ to $\mathcal{A}$.

  \item The morphism $f$ is an equivalence.
 \end{enumerate}
 Then the 2-category $\Lift(f; Y)$ is contractible. 
 \end{lemma}

The conditions of this lemma are satisfied in the case of the 2-functor $\cohLie \to \LieGrpd$. 
Assumption a) is standard (using the result mentioned at the start of Lemma~\ref{lemma:monoidal_adjoint_equivalence}), assumption b) by Theorem~\ref{thm:coherent_2-group_structure} (and a straightforward extension to the case of 2-arrows), and assumption c) by Theorem~\ref{thorem:smooth_adj_equiv}. 
Thus the 2-category of lifts of $\iota$, and hence of the equivalence data, is contractible.

\section{Comments on 2-representations of \texorpdfstring{$S(QG)$}{S(QG)}}

Just as groups act on vector spaces, 2-groups act nontrivially on \emph{2-vector spaces}, categories with extra linear structure and properties; we will call such actions \emph{2-representations}.
Note that the concept of a \emph{categorical representation} of an ordinary group is a special case. 
The easiest type of 2-vector space to understand is that of Kapranov--Voevodsky, which is a category equivalent to $\Vect^n$ for some $n$, where $\Vect$ is the category of vector spaces. 
More generally, one can take a finite groupoid $\Gamma$, and consider the category of $\Gamma$-actions on vector spaces, or the category of $\Gamma$-actions on just finite-dimensional vector spaces. 
Recall that such an action is an $\Gamma_0 := \Obj(\Gamma)$-indexed collection of vector spaces $\{V_x\}_{x\in\Gamma_0}$ such that for an arrow $a\colon x\to y$ there is associated a functorial isomorphism $V_x \to V_y$. In other words, one has a functor $\Gamma \to \Vect$, and the category of such actions is the category $\Vect^\Gamma := \Cat(\Gamma,\Vect)$ of functors, which is a protoypical 2-vector space.
For discussion of and references for the wide variety of 2-vector spaces in the literature, see \cite{nLab-2-vector-space}; we shall only mention specific examples in the current discussion.

If $\Gamma$ is additionally a 2-group then every $x\in \Gamma_0$ gives an autoequivalence functor $(-) \cdot x \colon \Gamma \to \Gamma$, giving an autoequivalence $\Vect^\Gamma \to \Vect^\Gamma$ by precomposition.
Similarly, an arrow $a\colon x\to y$ in $\Gamma$ gives a natural transformation of such precomposition functors.
In this way, the 2-group $\Gamma$ acts on the category $\Vect^\Gamma$.
It does not matter whether $\Gamma$ is a strict 2-group or not, in this construction.

We can further analyse what is happening in the case that the underlying groupoid of $\Gamma$ is an action groupoid $S//K$ for a finite group $K$ acting on a set $S$ (this is in fact true for any strict 2-group, via the construction using crossed modules).
For any $K$-representation $\rho\colon K \to GL(V)$ there is an $S//K$-action $\widetilde\rho\colon S//K \to \Vect$, given by sending an object $p\in S$ to the vector space $V$ and sending an arrow $p \xrightarrow{k} p\cdot k$ to the isomorphism $\rho(k)\colon V_p = V \to V = V_{pk}$.
Another way to view this is that $\widetilde\rho$ is the composition $S//K \to \ast//K \xrightarrow{\rho} \Vect$.
This gives a functor $\Rep(K) \to \Vect^{S//K}$, and in fact makes $\Rep(K)$ a subcategory of $\Vect^{S//K}$.
Moreover, the action of $S//K$ on $\Vect^{S//K}$ restricts to an action on $\Rep(K)$, since actions of $S//K$ of the form $\widetilde\rho$ are stable under precomposition by the functors $(-)\cdot p\colon S//K \to S//K$.
Moreover, the monoidal category structures are respected by the inclusion $\Rep(K) \to \Vect^{S//K}$ and by the action of $S//K$, namely direct sums and tensor products of representations, and of $S//K$-actions.
We have thus shown that $\Rep(K)$ carries a canonical 2-representation by $S//K$.

We can repeat the above analysis in the case that $\Gamma$ is a coherent Lie 2-group. 
In this setting, a $\Gamma$-action is a vector bundle on $\Gamma_0$ together with a continuous action by the arrows of $\Gamma$ on the vector bundle.
Categories of vector bundles are then 2-vector spaces (in algebraic geometry one uses variants such as coherent or quasicoherent sheaves).
For our case of interest, we need to be able to choose the kind of vector bundles we allow, for instance Hilbert bundles with a restricted structure group (as in \cite{BCMMS02}, for example).
Let us generically denote any choice of category of vector bundles on a Lie groupoid $X$ by $\Vect(X)$, with the assumption that it is closed under direct sums and some kind of tensor product, and includes all trivial bundles.
For those coherent Lie 2-groups that have as underlying Lie groupoid an action groupoid $K//S$ (as $S(QG)$ does, namely $QG//\wh{LG}$) we get a smaller 2-representation as above on the category $\Rep(K)$ of representations of $K$. 
Or rather, various smaller 2-representations depending on what category of ordinary representations one chooses to start with.

In the case of the string 2-group, we can consider the category $\Rep^+(\wh{LG})$ of positive energy representations of $\wh{LG}$, which admits a functor to any category $\Vect_{\Hilb}(S(QG))$ of Hilbert vector bundles on $S(QG)$, as described above. Thus $\Rep^+(\wh{LG})$ carries a canonical 2-representation of our model $S(QG)$ of the string 2-group. Note that the choice of what kind of Hilbert bundles on the Lie groupoid $S(QG)$ one uses does not matter here, since by assumption we always include trivial bundles. Additionally, the fusion product of positive energy representations should be preserved by the inclusion $\Rep^+(\wh{LG}) \to \Vect_{\Hilb}(S(QG))$, for a reasonable definition of fusion product of Hilbert bundles on $S(QG)$ (investigating which would take us too far afield for the current article).

Recall that the celebrated theorem of Freed--Hopkins--Teleman relates the Verlinde ring (which has as underlying additive group the representation ring of positive energy representations of $\wh{LG}$) to the twisted $G$-equivariant $K$-theory of $G$ (the action is by conjugation). 
A `twist', in this instance, can be represented by a gerbe, and the string group\footnote{there are other 2-group extensions of $G$, classified by $k\in H^4(BG,\mathbb{Z})$, constructed using the level-$k$ extension of $LG$ in place of $\wh{LG}$ in the previous sections. The discussion below applies to those 2-groups as well.} presents one such gerbe.
Note that we are only working in the special case of the theorem corresponding to a simply-connected compact Lie group as proved in \cite{FreedHopkinsTeleman13Loop-groups-and-twisted-K-theory-II}.

The work of Bouwknegt--Carey--Mathai--Murray--Stevenson \cite{BCMMS02} constructs a model of twisted $K$-theory using \emph{bundle gerbe modules}, which are a special sort of Hilbert bundle on the Lie groupoid underlying a bundle gerbe\footnote{in particular, the Hilbert bundles have a restricted structure group, and the $\ast//U(1)$ subgroupoid acts by scalar multiplication}. 
Any reasonable category of bundle gerbe modules on $S(QG)$ that is rich enough to calculate the $K$-theory of $G$ twisted by the gerbe $S(QG)$ could be used in the construction of a 2-representation as before.
In principle, one could use the treatment of equivariant bundle gerbes in \cite{MRSV15}---since $S(QG)$ is a 2-group extension it admits the structure of a $G$-equvariant bundle gerbe---to find categories of \emph{equivariant} bundle gerbe modules.
This will give a smaller 2-representation $\Vect_{\Hilb}(S(QG))_G$, but one that should still contain the 2-representation $\Rep^+(\wh{LG})$, since the assignment $\rho \mapsto \widetilde\rho$, under all reasonable definitions of equivariance, should result in an equvariant bundle gerbe module. Thus, at least conjecturally, the 2-representation $\Rep^+(\wh{LG})$ of our model of the string 2-group could be the `fixed points' under the adjoint action of $G$ on the 2-representation $\Vect_{\Hilb}(S(QG))$.

\end{document}